\documentclass[11pt]{article}

\usepackage[utf8]{inputenc}
\usepackage{amsmath}
\usepackage{amsfonts}
\usepackage{amssymb}
\usepackage{amsthm}
\usepackage{natbib}
\usepackage{color}
\usepackage[colorlinks=true,linkcolor=blue,citecolor=blue,pdfborder={0 0 0}]{hyperref}

\usepackage{dsfont}
\usepackage{bm}

\usepackage{graphicx}
\usepackage{enumerate}
\usepackage{paralist}

\theoremstyle{plain}
\newtheorem{theorem}{Theorem}[section]
\newtheorem{lemma}[theorem]{Lemma}
\newtheorem{proposition}[theorem]{Proposition}
\newtheorem{cor}[theorem]{Corollary}

\theoremstyle{definition}
\newcommand{\cond}{\textbf{C}\!\!}

\newtheorem{Hypo}{Hypothesis}

\theoremstyle{remark}
\newtheorem{remark}[theorem]{Remark}
\newtheorem*{notation*}{Notation}
\numberwithin{equation}{section}

\newcommand{\rset}{\mathbb{R}}

\newcommand{\ceil}[1]{\lceil{#1}\rceil}
\newcommand{\floor}[1]{\lfloor{#1}\rfloor}

\newcommand{\un}{\bm{1}}
\newcommand{\PP}{\mathbb{P}}
\newcommand{\esp}{\operatorname{\mathbb{E}}}
\newcommand{\var}{\operatorname{\mathbf{Var}}}
\newcommand{\PE}{P}
\DeclareMathOperator*{\argmin}{arg\,min}

\DeclareMathOperator{\tr}{tr}
\DeclareMathOperator{\Span}{span}
\DeclareMathOperator{\supp}{supp}
\DeclareMathOperator{\pen}{pen}
\DeclareMathOperator{\HS}{\mathcal{HS}}

\newcommand{\calH}{\mathcal{H}}

\usepackage{xspace}

\newcommand\ie{\emph{i.e.}\xspace }

\newcommand\eg{\emph{e.g.}\xspace }
\newcommand\iid{\ensuremath{\mathit{i.i.d.}}\xspace }

\newcommand\cdf{\emph{c.d.f.}\xspace}

\newcommand{\point}{\,\cdot\,}

\newcommand{\as}[1]{\textcolor{red}{[Anne: {#1}]}}

\newcommand{\sphere}{\mathbb{S}^{d-1}}
\newcommand{\cone}{\mathcal{C}}
\newcommand{\subsp}{V}
\newcommand{\proj}{\mathbf{\Pi}}
\newcommand{\ang}{\Theta}
\newcommand{\angf}{\theta}
\newcommand{\vertiii}[1]{{\left\vert\kern-0.25ex\left\vert\kern-0.25ex\left\vert #1
    \right\vert\kern-0.25ex\right\vert\kern-0.25ex\right\vert}}

\newcommand{\R}{\mathbb{R}}
\newcommand{\N}{\mathbb{N}}

\newcommand{\eps}{\varepsilon}
\newcommand{\VV}{\mathcal{V}}

\usepackage{authblk}

\begin{document}
\title{Principal Component Analysis for Multivariate Extremes}
\author[1]{Holger Drees}
\author[2]{Anne Sabourin}
\affil[1]{University of Hamburg, Department of Mathematics, Germany}
\affil[2]{LTCI, T\'{e}l\'{e}com Paris, Institut polytechnique de Paris, France}
\maketitle


\begin{abstract}

The first order behavior of
    multivariate heavy-tailed random vectors  above large radial thresholds is ruled by a limit measure in a regular variation
    framework. For a
    high dimensional vector, a reasonable assumption is that the
    support of this measure is concentrated on a lower dimensional
    subspace, meaning that certain linear combinations of the
    components are much likelier to be large than others. Identifying
    this subspace and thus reducing the dimension will facilitate a
    refined statistical analysis.  In this work we apply
  Principal Component Analysis (PCA) to a re-scaled version of
radially thresholded observations. 

  Within the statistical learning framework of empirical risk
  minimization, our main focus is to analyze the squared
  reconstruction error for the exceedances over large radial
  thresholds. We prove that the empirical risk converges to the true
  risk, uniformly over all projection subspaces. As a consequence, the
  best projection subspace is shown to converge in probability to the
  optimal one, in terms of the Hausdorff distance between their
  intersections with the unit sphere. In addition, if the exceedances
  are re-scaled to the unit ball, we obtain finite sample uniform
  guarantees to the reconstruction error pertaining to the estimated
  projection subspace. Numerical experiments illustrate the relevance
  of the proposed framework for practical purposes.\\~

  ~

  \noindent{\bf MSC primary} 62G32; {\bf secondary} 62H25. 
\end{abstract}

\section{Introduction}
\label{sec:intro}

If one wants to analyze the tail behavior of an $\R^d$-valued random vector $ X =(X^1,\ldots,X^d)$ one usually assumes that $X$ is regularly varying (if necessary after a standardization of the marginal distributions), \ie there exists a normalizing function $b$ and a non-zero 
measure $\mu$ on $\R^d\setminus\{0\}$ such that 
\begin{equation}
  \label{eq:rvalpha}
\mu_t(\point):=  (b(t))^{-1}\PP(X \in tB )\;\xrightarrow[t\to\infty]{} \mu(B)<\infty
\end{equation}
for all $\mu$-continuous Borel sets $B$ that are bounded away from the
origin. Equation~\eqref{eq:rvalpha} may be understood as a
generalization to arbitrary dimension of a heavy-tail assumption
regarding a real-valued random variable.  This mathematical
  framework is particularly useful in situations where the focus is on `tail events'
   of the kind $\{X\in B\}$ where the distance to
  the origin $u = \inf\{\|x\|: x \in B\}$ is large, for some norm
  $\|\cdot\|$. In a risk management context, the probability of such
  tail events is of crucial importance.  If the distance $u$ is
  so large that few or no data are available in the considered region,
  all attempts to resort to empirical estimation are in vain. One common idea behind statistical methods based on Extreme Value Theory (EVT) is to use a small proportion of the available data (those with a comparatively large norm) to learn an estimate for $\mu$, to be used for quantifying the probability of tail events.

\subsection{Regular Variation}
A substantial reference concerning the probabilistic  aspects of regular variation in the setting of EVT is~\cite{resnick2013extreme}, see also  \cite{resnick2007heavy} for application-oriented examples. Regular variation for Borel measures on  Polish spaces has since  been revisited in \cite{hult2006regular}  and \cite{lindskogetal14}.
It is well known that if Equation~\eqref{eq:rvalpha} holds true, then the limit measure $\mu$ is homogeneous of order $-\alpha$ for some $\alpha>0$. Moreover, the normalizing function $b$ and the norm $\|X\|$ are regularly varying, too: $b(tx)/b(t)\to x^{-\alpha}$ and $\PP\{\|X\|>tx\}/ \PP\{\|X\|>t\}\to x^{-\alpha}$ as $t\to\infty$ for all $x>0$. Here $\|\cdot\|$ may be any norm on $\R^d$, but in what follows we only consider the Euclidean norm.

Because the limit measure is homogeneous, after a polar transformation, it can be decomposed into a so-called spectral (or angular)  probability measure $H$ and an independent radial component, that is
\begin{equation} \label{eq:spectraldef}
\mu\Big\{x\in \R^d \;:\; \|x\|>r, \frac{x}{\|x\|} \in A\Big\} = cr^{-\alpha}  H(A),
\end{equation}
for all $r>0$ and all Borel subsets $A$ of the unit sphere, with $c:= \mu\{x\;:\;\|x\|>1\}$.
Whereas the literature is plentiful concerning the design and the asymptotics of  flexible multivariate parametric or non-parametric models for $\mu$ or integrated versions of it (see \eg \cite{segers2012,fougeres2015bias,einmahl2001nonparametric,genest2009rank,rootzen2018multivariate2}, or \cite{beirlant2006statistics} and the references therein), the issue of how to escape the curse of dimensionality has only recently been raised (see below). One reason for this may be  that a major application of EVT is  environmental, spatial extremes such as heavy rainfalls, heat waves, droughts or  floods.  In this context, max-stable or generalized Pareto  spatial models are widely used (\cite{padoan2010likelihood,ferreira2014generalized,schlather2002models}) which have built in a priori information about the spatial dependence structure, thus reducing the effective dimension.

\subsection{Dimensionality  reduction for extreme values, a brief overview}

For applications such as \emph{e.g.} anomaly detection or network monitoring  where no particular structure is known a priori,  dimensionality reduction suggests itself as a preliminary step before implementing any kind of learning procedure and
  the subject   is recently receiving increasing attention.   If  $d$ is moderate or large, the measure $\mu$ (and hence $H$) will often exhibit some `sparse' structure.  For example, if some of the components of $X$ are asymptotically independent, \ie for some index set $I\subset \{1,\ldots,d\}$ of size $|I|\in\{1,\ldots, d-1\}$
$$ \PP\Big\{\max_{i\in I} |X^i|>t,\max_{i\not\in I} |X^i| > t\Big\} = o\Big(\PP\Big\{\max_{1\le i\le d} |X^i|>t\Big\}\Big),
$$
then $\mu$ is concentrated on $\{x\in\R^d\mid \max_{i\in I}|x_i|=0 \text{ or } \max_{i\not\in I}|x_i|=0\}$.
More generally, one may consider the case where only a small number of subsets of components $\{I_k \subset \{1,\ldots, d\}, k = 1,\ldots, K  \}$ are likely to be large simultaneously, while the other components remain small. Here, `small number' is understood relatively to the $2^d -1$ non empty possible subsets of components. This setting applies \emph{e.g.} to heavy rainfalls in a spatial setting (storms are usually localized, so that neighboring sites are more likely to be concomitantly impacted) or of shocks over different assets of a financial portfolio. \cite{chautru2015dimension} proposes a clustering approach combined with spherical data analysis to detect structures of this type.
\cite{goix2016sparse,goix2017sparse}
propose an  algorithm with moderate computational cost (linear in the dimension and the sample size) and finite sample uniform guarantees. Their error bounds are linear in $d$ and scale as $1/\sqrt{k}$, where $k$ is the number of order statistics of each component which are considered extreme during the training step. A refinement of the latter framework is proposed in the yet unpublished work of \cite{simpson2018determining}.  \cite{chiapino2016feature} and  \cite{chiapino2019identifying} aim at identifying subgroups of components for which the  probability of a joint excess over a large quantile is not negligible compared to that of an excess by a single component. \cite{EngelkeHitz2018} use graphical models to reduce the complexity of the extremal dependence structure.  In a regression context, \cite{gardes2018tail} sets up a mathematical framework for tail dimension reduction suited to the case where the distribution of the target variable above high thresholds  only depends on the projection of the covariates on a lower dimensional subspace. Consistency of K-means clustering applied to the most extreme observations of a data set  has recently been proven in the unpublished work of~\cite{janssen2019k}.

 \subsection{Principal component analysis (PCA)  and support identification}
 Here we focus on finding a linear subspace on which $\mu$ is (nearly) concentrated. In a classical setting, when $\|X\|$ has finite second moments, PCA (\cite{anderson1963asymptotic}) is the method of choice to determine such supporting linear subspaces if \iid random vectors $X_i$, $1\le i\le n$, with the same distribution as $X$ are observed.  Theoretical guarantees obtained  so far concern  the   reconstruction error (\cite{koltchinskii2000random,shawe2005eigenspectrum,blanchard2007statistical,koltchinskii2016new}) or the approximation error for the eigenspaces of the covariance matrix (\cite{zwald2006convergence}), under the assumption that the sample space (or the feature space for Kernel-PCA) has finite diameter or that sufficiently high order moments exist.

For motivation of our version of PCA, it is useful to keep the following working hypothesis in mind, although it is not required for most results to hold:
\begin{Hypo}\label{hypo:sparse}
The vector space $V_0 = \Span(\supp \mu)$  generated by the support of $\mu$ has dimension  $p< d$.
\end{Hypo}
Note that then the points $(X_i/t) \un\{\|X_i\|>t\}$ are more and more concentrated on a neighborhood of $V_0$ as $t$ increases, but that usually they will not lie on $V_0$. If the dimension $p$ of $V_0$ is known, then it suggests itself to approximate $V_0$ by the subspace of dimension $p$ which is `closest' in expectation to these points.

In PCA one measures the closeness by the squared Euclidean distance which hugely alleviates the optimization problem as one may work with orthogonal projections in  the Hilbert space $L_2$.
However, this approach assumes finite second moments that cannot be taken for granted in the above setting. Indeed, if $\alpha<2$ then $\esp(\|X_i\|^2)=\infty$. Hence, we will instead consider re-scaled vectors
\begin{equation}  \label{eq:rescaling}
  \Theta_i := \omega(X_i)X_i, \quad 1\le i\le n,
\end{equation}
where $\omega:\R^d\to (0,\infty)$ is a suitable scaling function. The most common choice is $\omega(x)=1/\|x\|$, leading to $\Theta_i$ on the unit sphere, and we will focus on this re-scaling when we derive finite sample bounds on the reconstruction error (see Section \ref{sec:recover}). However, consistency results will be proved for considerably more general scaling functions; cf.\ Section \ref{sec:probaBackground}.

To the best of our knowledge, the only existing work considering PCA properly speaking for high dimensional extremes is the unpublished paper of \cite{cooley2016decompositions}. The authors discuss a transformation mapping negative observations to small positive ones and apply PCA in this transformed space. They also use a preliminary re-scaling involving  the norm of the transformed vector. They illustrate their approach with simulations and real data examples, without deriving theoretical statistical guarantees.

\subsection{Notation and risk minimization setting}
To give a formal description of our method, we first introduce some notation. All random variables are defined on some probability space $(\mathcal{X},\mathcal{A},\PP)$; the expectation with respect to $\PP$ is denoted by $\esp$.
For $x\in\R^d$ and $t>0$, let
\begin{equation} \label{eq:angnot}
  \begin{aligned}
    \angf(x) & =  \omega(x) x ,\\
    \angf_t(x) & =  \omega(x) x \un\{\|x\|>t\},\\
    \ang &= \theta(X) = \omega(X) X,\\
    \ang_t &  =\angf_t(X) =  \ang\un\{\|X\|  > t\}.
  \end{aligned}
\end{equation}

By  $P$ we denote the distribution of $X$ and by $P_t$ its conditional distribution given that $\|X\|>t$, i.e. $P_t(\cdot)=\PP(X\in\cdot \mid \|X\|>t\}$. Then $P_\infty:=\mu/\mu((B_1(0))^c)$  is the weak limit of $P_t(t\cdot)$ (with $B_1(0)$ denoting the closed unit ball); cf.\ \eqref{eq:rvalpha}.

For any probability measure $Q$ and any $Q$-integrable function  $f$, we denote the expectation of $f$ with respect to $Q$ by $Qf$ or $Q(f)$. By $\esp_t$ we denote the conditional expectation (with respect to $\PP$) given $\|X\|>t$ so that $\esp_t(f(X))=P_t(f)$, provided the expectations exist.

For any linear subspace $V\subset\R^d$, let  $\proj_V$ be the  orthogonal projection onto $V$
(or the associated projection  matrix), and let $\proj_V^\perp$ be the orthogonal projection onto the orthogonal complement $V^\perp$  of $V$.

To apply PCA to the re-scaled vectors, we have to assume that the scaling function $\omega$ is chosen such that $\esp(\|\Theta\|^2)=P(\|\theta\|^2)<\infty$ and $P_\infty(\|\theta\|^2)<\infty$. Note that this condition is always fulfilled if there exist $\beta>1-\alpha/2$ and $c>0$ such that $\omega(x)\le c\|x\|^{-\beta}$ for all $x\in\R^d$. For simplicity's sake, in what follows we will impose the following stronger homogeneity condition:
\begin{equation} \label{eq:omegacond}
 \exists\,\beta\in \Big(1-\frac\alpha 2,1\Big]\;\forall\,\lambda>0, x\in\R^d:\quad   \omega(\lambda x) = \lambda^{-\beta} \omega(x) \quad \text{and} \quad c_\omega:=\sup_{x\in\sphere} \omega(x)< \infty,
\end{equation}
where $\sphere:=\{x\in\R^d\,:\, \|x\|=1\}$ denotes the unit sphere. Note that then $\|\angf(x)\|\le c_\omega\|x\|^{1-\beta}$. The choice $\omega(x)=\|x\|^{-\beta}$ seems natural, but different choices allow for focusing on particular aspects of the extreme value behavior. For instance, if one is only interested in the positive components of $X$, one may choose $\omega(x)=\|x\|^{-\beta}\un_{[0,\infty)^d}(x)$.

Hypothesis 1 is equivalent to the statement that $\inf_{V:\dim(V)=p} R_\infty(V)=0$ and $\inf_{V:\dim(V)=p'} R_\infty(V)>0$ for all $p'<p$ where
$$  R_\infty(V):= P_\infty \|\proj_V \angf-\angf\|^2  = P_\infty \|\proj_V^\perp \angf\|^2
$$
and the infima are taken over all linear subspaces of the specified dimension. The risk $R_\infty$ may be interpreted as the expected reconstruction error in the limit model if the re-scaled observation $\ang$ is replaced with its lower dimensional approximation $\proj_V\ang$.
Since $P_t(t\cdot)\to P_\infty(\cdot)$ weakly, one may approximate $V_0$ by a subspace $V_t^*=V_t^{p*}$ of dimension $p$ which minimizes the conditional risk
\begin{equation}
  \label{eq:finiteRisk}
  R_t(\subsp):=  P_t\big(\| \proj_V^\perp{\angf}) \|^2 \big) =  \esp_t\big( \| \proj_V^\perp{\ang} \|^2\big)
\end{equation}
 given that $\|X\|$ exceeds a high threshold $t>0$. Note that $V_t^*$ may be of interest even if Hypothesis 1 only holds approximately, in the sense that $P_\infty$ concentrates most of its mass on a small neighborhood of a $p$-dimensional subspace.

It is natural to `estimate' $V_t^*$ (and thus $V_0$) by a minimizer of the corresponding empirical risk
$$ \hat R_t(V) := \frac 1{N_t} \sum_{i=1}^n \|\proj_V^\perp \Theta_i\|^2 \un\{\|X_i\|>t\}\quad \text{with}\quad N_t :=\sum_{i=1}^n \un\{\|X_i\|>t\}. $$
 Here the threshold $t$ must be chosen suitably, depending on the sample size. To this end, often order statistics of the norms of the observed vectors are used, and we follow this approach. Let $X_{(j)} = X_{\sigma(j)}$ where $\sigma$ is a permutation of indices such that
$\|X_{(1)}\|\ge \|X_{(2)}\|\ge \dotsb \ge \|X_{(n)}\|$. (For brevity, we suppress the dependence on $n$ in our notation of order statistics.)
For  $1\le k\le n$, denote by $\hat{t}_{n,k} = \| X_{(k+1)}\|$ the empirical quantile of level $1-k/n$ for $\|X\|$. We define the empirical risk for the subspace $V$ related to the $k$ largest observations as
 \begin{equation}
   \label{eq:empiricalRisk}
   R_{n,k} (V)
   = \frac{1}{k}\sum_{i=1}^n \| \proj_V^\perp{\ang}_{i,\hat{t}_{n,k}} \|^2
 \end{equation}
 where $\ang_{i,t} =  \angf_t(X_{i})$ in accordance with the notation introduced in \eqref{eq:angnot}. A minimizer of $ R_{n,k} (V)$ among all linear subspaces of dimension $p$ will be denoted by $\hat V_n = \hat V_n^{p}$.
 It is the main goal of the present paper to analyze the asymptotic and the finite sample behavior of the empirical risk $R_{n,k} (V)$ and its minimizer $\hat V_n$.

 \subsection{Outline}

 In Section 2 we will first show that the minimizer of the risk $R_t$ based on a finite threshold $t$ converges to the minimizer of the limit risk $R_\infty$, and thus under Hypothesis 1 to $V_0$, as $t\to\infty$. Moreover, we show consistency of the empirical minimizer $\hat V_n$ under condition \eqref{eq:omegacond} with suitable $\beta$. In Section 3, we derive non-asymptotic uniform bounds on  $|R_{n,k}(V)-R_{t_{n,k}}(V)|$  and $|\hat R_t(V)-R_t(V)|$ for the most important scaling $\omega(x)=1/\|x\|$. Furthermore, we construct uniform confidence bands for $R_t(V)$. The results obtained in a simulation study are reported in Section 4. In particular, we explore the choice of the dimension $p$ based on empirical risk plots and the effect of a PCA projection on estimators of probabilities expressed in terms of the spectral measure $H$. Finally, Section 5 contains some details about the proof of a modification of a result by \cite{blanchard2007statistical}.

\section{Consistency of risk minimizers}\label{sec:probaBackground}

In this section we first discuss how to calculate minimizers of the conditional risk $R_t$ given $\|X\|>t$ and the empirical risk $R_{n,k}$. Moreover, we prove that these converge in some sense towards a minimizer of $R_\infty$.

It is well known that a point of minimum of  $V\mapsto\esp\|\proj_V^\perp Y\|^2$ can be derived from the spectral analysis of the matrix of second (mixed) moments of $Y$:
\begin{lemma}
  \begin{enumerate}
    \item Let $Y$ be an $\R^d$-valued random vector with $\esp(\|Y\|^2)<\infty$ and $\Sigma:=\esp(YY^\top)$. Let $\lambda_1\ge \lambda_2\ge\cdots\ge\lambda_d\ge 0$ denote the eigenvalues of $\Sigma$ with corresponding orthogonal eigenvectors $x_1,\ldots,x_d$. Then $V^*=\Span(x_1,\ldots,x_p)$ minimizes $\esp(\|\proj_V^\perp Y\|^2)$ among all linear subspaces $V$ of dimension $p$. In the case $\lambda_p>\lambda_{p+1}$ it is the unique minimizer.
    \item If the scaling condition \eqref{eq:omegacond} holds and $\lambda_1\ge \lambda_2\ge\cdots\ge\lambda_d\ge 0$ denote the eigenvalues of $\Sigma_t:=\esp_t(\ang \ang^\top)$ with corresponding orthogonal eigenvectors $x_1,\ldots,x_d$, then $V^*=\Span(x_1,\ldots,x_p)$ minimizes $R_t(V)$ among all linear subspaces $V$ of dimension $p$. In the case $\lambda_p>\lambda_{p+1}$ it is the unique minimizer.
    \item If the scaling condition \eqref{eq:omegacond} holds and $\lambda_{n,1}\ge \lambda_{n,2}\ge\cdots\ge\lambda_{n,d}\ge 0$ denote the eigenvalues of $\Sigma_{n,k}:=k^{-1} \sum_{i=1}^n({\ang}_{i,\hat{t}_{n,k}}{\ang}_{i,\hat{t}_{n,k}}^\top)$ with corresponding orthogonal eigenvectors $x_{n,1},\ldots,x_{n,d}$, then $\hat V_n=\Span(x_{n,1},\ldots,x_{n,p})$ minimizes $R_{n,k}(V)$ among all linear subspaces $V$ of dimension $p$.
    \end{enumerate}
\end{lemma}
A proof of assertion (i) can e.g.\ be found in \cite{seber84}, Theorem 5.3, where also other optimality properties of the minimizers are given. Both the other results follow directly by an application of (i) with $Y$ equal to $\ang$ conditional on $\|X\|>t$, respectively a random variable according to the empirical distribution of the ${\ang}_{i}$ for which $\|X_i\|>\hat{t}_{n,k}$. If $\lambda_p=\lambda_{p+1}$, then the minimizer is not unique. With $m=\min\{i\in\{1,\ldots,p\} \;:\; \lambda_i=\lambda_p\}$ any minimizer $V_t^*$ of $R_t$ can be represented as $V_t^*=\Span(x_1,\ldots,x_{m-1},\tilde x_m,\ldots,\tilde x_p)$ where $\tilde x_m,\ldots,\tilde x_p$ are orthogonal eigenvectors to the eigenvalue $\lambda_p$ and all these subspaces are minimizers. An analogous statement holds for the empirical risk.

Next we discuss the relationship between $R_t$ and $R_\infty$ and their respective minimizers.
The convergence of the risks is an immediate consequence of the following simple lemma.
\begin{lemma} \label{lem:intconv}
 Let $f:\R^d\to \R$ be a measurable function that is locally bounded, $P_\infty$-a.e.\ continuous and satisfies $\limsup_{\|x\|\to\infty} |f(x)|\|x\|^{-\tilde\alpha}<\infty$ for some $\tilde\alpha<\alpha$. Then $\lim_{t\to\infty} \int f(x/t)\, P_t(dx) = \int f(x)\, P_\infty(dx)$.
\end{lemma}
\begin{proof}
  According to \eqref{eq:rvalpha}, $P_t(t\point)=\PP(X\in t\point\mid \|X\|>t)\to \mu(\point)/\mu((B_1(0))^c)= P_\infty(\point)$ weakly. Let $Y_t$ be a random vector with distribution $P_t(t\point)$ and $Y_\infty$ a random vector with distribution $P_\infty$. Since $\int f(x/t)\, P_t(dx)=\esp f(Y_t)$, the assertion follows if the $f(Y_t)$ are asymptotically uniformly integrable (see \cite{van2000asymptotic}, Theorem 2.20).

  By assumption $f(Y_t)$ can be bounded by a multiple of $1+\|Y_t\|^{\tilde\alpha}$. Now, for all $\tau\in [0,\alpha)$ and $t\ge t_0$ for some sufficiently large $t_0$, integration by parts, regular variation of $u\mapsto u^{\tau-1}\PP\{\|X\|>u\}$ and Karamata's theorem yield
  \begin{align}
   \esp \|Y_t\|^{\tau} & = \int \|x/t\|^{\tau}\, P_t(dx)\nonumber\\
   & = \frac{t^{- \tau}}{\PP\{\|X\|>t\}} \int_t^\infty u^{\tau}\, \PP^{\|X\|}(du)\nonumber\\
   &  = \frac{t^{- \tau}}{\PP\{\|X\|>t\}} \tau \int_t^\infty u^{\tau-1}\PP\{\|X\|>u\}\, du\nonumber\\
   & \le 2 \frac{t^{- \tau}}{\PP\{\|X\|>t\}} \tau \frac{t^{\tau}\PP\{\|X\|>t\}}{\alpha-\tau}\nonumber\\
   & = 2\frac{\tau }{\alpha-\tau}.   \label{eq:mombound}
  \end{align}
  In particular, $\sup_{t\ge t_0}\esp\|Y_t\|^{\tilde\alpha(1+\eps)}<\infty$ for $\eps\in (0,\alpha/\tilde\alpha-1)$, so that $\|Y_t\|^{\tilde\alpha}$ and thus $f(Y_t)$ are asymptotically uniformly integrable.
\end{proof}

\begin{cor}\label{cor:Rtconverges}
Suppose that $\omega$ fulfills condition \eqref{eq:omegacond}. Then,  for any subspace $V$ of $\rset^d$, the suitably standardized associated finite threshold risk converges:
\[ \lim_t t^{2(\beta-1)} R_t(V)   = R_\infty(V). \]
\end{cor}
\begin{proof}
  Note that by the homogeneity of $\omega$,
  $$ t^{2(\beta-1)}R_t(V)= P_t(\|\proj_V^\perp t^{\beta-1}\angf\|^2) =\int f(x/t)\, P_t(dx)$$
  with $f(x):=\|\proj_V^\perp\angf(x)\|^2=\|\proj_V^\perp\omega(x)x\|^2\le c_\omega^2\|x\|^{2(1-\beta)}$. Since $2(1-\beta)<\alpha$, Lemma~\ref{lem:intconv} yields the assertion.
\end{proof}

In view of Corollary \ref{cor:Rtconverges}, one may ask whether a minimizer of $\tilde R_t:=t^{2(\beta-1)}R_t$ (which of course is also a minimizer of $R_t$) converges in some sense to a minimizer of $R_\infty$.
  Denote by $\mathcal{V}_p$ the set of all subspaces of $\rset^d$ of dimension $p$, endowed with the metric $\rho(V,W) = \vertiii{\proj_V - \proj_W} = \vertiii{\proj_V^\perp - \proj_W^\perp}=\sup_{x\in\sphere}\|\proj_V^\perp x- \proj_W^\perp x\|$, where $\vertiii{\point}$ denotes the operator norm.  

\begin{remark} \label{rem:hausdorffdistance}
  Note that $\rho(V,W)$ also gives an upper bound on the  Hausdorff distance between $V\cap \sphere$ and $W\cap\sphere$. To see this, let $x^*\in V\cap\sphere$ and $y^*\in W\cap\sphere$  be such that the Hausdorff distance equals $\inf_{y\in W\cap\sphere}\|x^*-y\| = \|x^*-y^*\|$. Then $y^*=\proj_W x^*/\|\proj_Wx^*\|$, $\|x^*-\proj_W x^*\|\le\rho(V,W)$ and $\|\proj_W x^*\|^2\ge 1-(\rho(V,W))^2$. Hence
  \begin{align*}
    \|x^*-y^*\|^2 & =  \|x^*-\proj_Wx^*\|^2 + \|\proj_W x^*-y^*\|^2\\
    & \le (\rho(V,W))^2+(1-\|\proj_Wx^*\|)^2\\
    & \le (\rho(V,W))^2+\Big(1-\sqrt{1-(\rho(V,W))^2}\Big)^2\\
    & = 2\Big(1-\sqrt{1-(\rho(V,W))^2}\Big).
  \end{align*}
\end{remark}

  \begin{theorem}\label{th:consistency_Vtstar}
  Suppose that $\omega$ satisfies condition \eqref{eq:omegacond} and that $R_\infty$ has a unique minimizer $V_\infty^*$ in $\VV_p$. Then for any minimizer $V_t^*$ of $R_t$ in $\VV_p$ one has
   \[
      \lim_{t\to\infty} \rho(V_t^*,V_\infty^*)=0.
   \]
  \end{theorem}
  The following lemma plays a crucial role in the proof of Theorem \ref{th:consistency_Vtstar}.
  \begin{lemma} \label{lem:equicont}
    If $\omega$ satisfies condition \eqref{eq:omegacond}, then for sufficiently large $t_0$, the standardized risks $\tilde R_t=t^{2(\beta-1)}R_t$, $t\ge t_0$, are equicontinuous w.r.t.\  $\rho$.
  \end{lemma}
  \begin{proof}
    First note that $\big|\|\proj_V^\perp \angf(x)\|-\|\proj_W^\perp\angf(x)\|\big|\le \|\proj_V^\perp\angf(x)-\proj_W^\perp\angf(x)\| \le \|\angf(x)\|\rho(V,W) \le c_\omega\|x\|^{1-\beta} \rho(V,W)$. Choose $t_0$ as in the proof of Lemma \ref{lem:intconv} and recall the definition of $Y_t$ given there. Then, by \eqref{eq:mombound}, for all subspaces $V,W$ of $\R^d$
    \begin{align*}
       |\tilde R_t(V)-\tilde R_t(W)| & = t^{2(\beta-1)}\Big| P_t \|\proj_V^\perp \angf\|^2-P_t \|\proj_W^\perp \angf\|^2\Big| \\
       & \le t^{2(\beta-1)} P_t \Big(\big|\|\proj_V^\perp \angf\|- \|\proj_W^\perp \angf\|\big|\point\big(\|\proj_V^\perp\angf\|+\|\proj_W^\perp\angf\|\big)\Big) \\
       & \le 2 t^{2(\beta-1)}P_t\|\angf\|^2\rho(V,W)\\
       & \le 2c_\omega^2 \esp\|Y_t\|^{2(1-\beta)}\rho(V,W)  \\
       & \le 2 c_\omega^2 \frac{4(1-\beta)}{\alpha-2(1-\beta)} \rho(V,W)
    \end{align*}
    which proves the assertion.
\end{proof}

  \begin{proof}[Proof of Theorem \ref{th:consistency_Vtstar}]

  We first prove that $\VV_p$ is compact w.r.t.\ $\rho$. The assertion then follows by standard arguments using Lemma \ref{lem:equicont}.

     We have to show that any sequence $(V_n)_{n\in\N}$ in $\mathcal{V}_p$ has a convergent subsequence. For each $n$, let $(u_{1,n}, \ldots, u_{p,n})$  be an orthonormal basis for $V_n$ so that $\proj_{V_n} x = U_n U_n^\top x$ where $U_n$ denotes the matrix with columns $u_{j,n}$. The vectors $(u_{j,n})_{1\le j\le p}$ belong to the compact set $(\sphere)^p$. Thus there exists a subsequence $n_\ell$ such that $u_{j, n_\ell}\to u_j^0$ for all $1\le j\le p$. Since for all $n$, $\langle u_{j,n}, u_{i,n}\rangle = \delta_{i,j}$, we also have   $\langle u_{j}^0, u_{i}^0\rangle = \delta_{i,j}$ and the $u_{j}^0, j\le p$, form an orthonormal family in $\rset^d$.  Let $V^0$ be the  space generated by the $u_j^0$'s and denote by $U_0$ the matrix with these columns. Then $V^0$ has dimension $p$, i.e.\ $V^0\in\mathcal{V}_p$, and by construction
\[\rho(V_{n_\ell}, V^0)=\sup_{x\in\sphere}\|(U_{n_\ell}U_{n_\ell}^\top - U_0U_0^\top)x\| \to 0.\]
which proves the claimed compactness.

Now assume that the assertion of the theorem was wrong. By the compactness of $\VV_p$, then there exist a sequence $t_n\to\infty$ such that $V_{t_n}^*$ converges to some $V_\infty\ne V_\infty^*$. By Lemma \ref{lem:equicont}, $|\tilde R_{t_n}(V_{t_n}^*)-\tilde R_{t_n}(V_\infty)|\to 0$, and by
Corollary \ref{cor:Rtconverges} $|\tilde R_{t_n}(V_\infty)-R_\infty(V_\infty)|\to 0$ and $|\tilde R_{t_n}(V_\infty^*)-R_\infty(V_\infty^*)|\to 0$. Hence, for $\eps:=R_\infty(V_\infty)-R_\infty(V_\infty^*)$, that is strictly positive by assumption, and sufficiently large $n$, one may conclude a contradiction:
$$ R_\infty(V_\infty)\le \tilde R_{t_n}(V_\infty)+\frac\eps 4 \le \tilde R_{t_n}(V_{t_n}^*)+\frac\eps 2\le \tilde R_{t_n}(V_\infty^*)+\frac\eps 2\le R_\infty(V_\infty^*)+\frac{3\eps}4<R_\infty(V_\infty).
$$
Therefore, the assertion must be correct.
\end{proof}

Under Hypothesis 1, $V_0$ is the unique minimizer of $R_\infty$ over $\VV_p$, that is if we minimize the risk over linear subspaces with the correct dimension, as the following result shows. Hence in this case, $V_t^*$ converges to $V_0$.
\begin{lemma}\label{lem:V0Vstar}
  Under Hypothesis 1, for any subspace $V\subset \rset^d$ of arbitrary dimension one has
  \[
R_\infty(V) = 0 \iff  V_0 \subset V.
\]
Thus, $V_0$ is the unique minimizer of $R_\infty$ in $\VV_p$, whereas on $\VV_{\tilde p}$ with $\tilde p>p$ the points of minimum of the limit risk $R_\infty$ are not unique.
  \end{lemma}
  \begin{proof}

    If $ V_0\subset V$ then $V^\perp \subset V_0^\perp $. By Hypothesis 1, $P_\infty$ is concentrated on $V_0$, which implies
    $R_\infty(V)=P_\infty\|\proj_V^\perp {\angf}\|^2 = 0$.

    Conversely, if $R_\infty(V)=0$, then $1=P_\infty\{\proj_V^\perp \angf=0\}=P_\infty(V)$. By definition of $P_\infty$ and the homogeneity of $\mu$, this means that the support of $\mu$ must be a subset of $V$ and thus $V_0\subset V$.
\end{proof}

In the remaining part of this section, we will establish analogous consistency results for the empirical risk $R_{n,k}$ and its minimizer. In what follows, let $F_{\|X\|}$ be the \cdf of $\|X\|$, $F^\leftarrow_{\|X\|}$ its generalized inverse (quantile function) and define
$$ t_{n,k} := F^\leftarrow_{\|X\|}(1-k/n). $$
We start with consistency of the standardized empirical risk.
\begin{proposition} \label{prop:empriskconsist}
  If $\omega$ satisfies condition \eqref{eq:omegacond}, then $t_{n,k}^{2(\beta-1)}R_{n,k}(V)\to R_\infty(V)$ in probability for all linear subspaces $V$ of $\R^d$.
\end{proposition}
\begin{proof} For simplicity, we assume that $F_{\|X\|}$ is continuous in the tail (so that there are no ties among the observed norms), but the proof can be easily generalized using standard techniques from the theory of regular varying functions.
  First we want to replace the random threshold $\hat t_{n,k}$ with $t_{n,k}$  in the definition of $R_{n,k}$. Observe that by the H\"{o}lder inequality
  \begin{align*}
    t_{n,k}^{2(\beta-1)}& \Big|R_{n,k}(V)-\frac 1k\sum_{i=1}^n \|\proj_V^\perp\ang_i\|^2\un\{\|X_i\|>t_{n,k}\}\Big|\\
    & \le \frac 1k t_{n,k}^{2(\beta-1)} \sum_{i=1}^n \|\proj_V^\perp\ang_i\|^2 \big|\un\{\|X_i\|>\hat t_{n,k}\}- \un\{\|X_i\|>t_{n,k}\}\big|\\
    & \le \Big[ \frac 1k \sum_{i=1}^n t_{n,k}^{(2+\eta)(\beta-1)} \|\ang_i\|^{2+\eta} \un\{\|X_i\|>t_{n,k}\wedge \hat t_{n,k}\}\Big]^{2/(2+\eta)}\\
     & \quad \point  \Big[\frac 1k \sum_{i=1}^n \big|\un\{\|X_i\|>\hat t_{n,k}\}- \un\{\|X_i\|>t_{n,k}\}\big|^{(2+\eta)/\eta}\Big]^{\eta/(2+\eta)}.
  \end{align*}
  where $\eta>0$ is chosen such that $(2+\eta)(1-\beta)<\alpha$.

  It is well known that $\hat t_{n,k}/t_{n,k}\to 1$ in probability. Thus there exists a sequence $\delta_n\downarrow 0$ such that $P\{\hat t_{n,k}>t_{n,k}(1-\delta_n)\}\to 0$. By \eqref{eq:mombound} and the regular variation of $1-F_{\|X\|}$
  \begin{align*}
  \esp& \Big(t_{n,k}^{(2+\eta)(\beta-1)}\|\ang\|^{2+\eta}\un\{\|X_i\|>t_{n,k}(1-\delta_n)\}\Big)\\
   & \le c_\omega^{2+\eta} \esp \|Y_{t_{n,k}(1-\delta_n)}\|^{(2+\eta)(1-\beta)} (1-F_{\|X\|}(t_{n,k}(1-\delta_n))) \\
   & = O(1-F_{\|X\|}(t_{n,k})) = O(k/n).
  \end{align*}
  In particular, $k^{-1} \sum_{i=1}^n t_{n,k}^{(2+\eta)(\beta-1)} \|\ang_i\|^{2+\eta} \un\{\|X_i\|>t_{n,k}\wedge \hat t_{n,k}\}$ is stochastically bounded.

  Furthermore,
  $$ \sum_{i=1}^n \big|\un\{\|X_i\|>\hat t_{n,k}\}- \un\{\|X_i\|>t_{n,k}\}\big|^{(2+\eta)/\eta} = \Big|\sum_{i=1}^n \un\{\|X_i\|>t_{n,k}\}-k\Big|,
  $$
  because there exist exactly $k$ exceedances of $\hat t_{n,k}$, and either all non-vanishing differences of the indicator functions equal 1 or all equal $-1$, depending on whether $\hat t_{n,k}<t_{n,k}$ or $\hat t_{n,k}>t_{n,k}$. Now, the last sum is binomially distributed with parameters $n$ and $k/n$. By the central limit theorem for triangular arrays, the right hand side is of the stochastic order $k^{1/2}$.

  A combination of these results show that
  \begin{equation}  \label{eq:replacerandthresh}
     t_{n,k}^{2(\beta-1)} \Big|R_{n,k}(V)-\frac 1k\sum_{i=1}^n \|\proj_V^\perp\ang_i\|^2\un\{\|X_i\|>t_{n,k}\}\Big| = O_P\big(  k^{-\eta/(2(2+\eta))}\big) = o_P(1)
  \end{equation}
  uniformly for all subspaces $V$.

  In view of Corollary \ref{cor:Rtconverges}, it thus suffices to show that
  \begin{align*}
    t_{n,k}^{2(\beta-1)}& \Big| \frac 1k \sum_{i=1}^n \|\proj_V^\perp\ang_i\|^2 \un\{\|X_i\|>t_{n,k}\} - R_{t_{n,k}}(V)\Big| \\
     \le & t_{n,k}^{2(\beta-1)} \Big| \frac 1k\sum_{i=1}^n \Big( \|\proj_V^\perp\ang_i\|^2 \un\{\|X_i\|\in (t_{n,k},d_{n,k}]\} - \esp\big(\|\proj_V^\perp\ang_i\|^2 \un\{\|X_i\|\in (t_{n,k},d_{n,k}]\}\big)\Big)\Big|\\
    & \quad + t_{n,k}^{2(\beta-1)} \Big| \frac 1k \sum_{i=1}^n \Big(\|\proj_V^\perp\ang_i\|^2 \un\{\|X_i\|>d_{n,k}\} - \esp\big(\|\proj_V^\perp\ang_i\|^2 \un\{\|X_i\|>d_{n,k}\}\big)\Big)\Big|\\
     =: & T_{n,1}+T_{n,2} \to 0
  \end{align*}
  in probability, with $d_{n,k}:= (\log k) t_{n,k}$.

  Let $\alpha^*:= 4(1-\beta)\vee (\alpha+1)$. By similar calculation as in the proof of Corollary \ref{cor:Rtconverges}
  \begin{align*}
    \esp(T_{n,1}^2) & = \var(T_{n,1})\\
     & \le  \frac n{k^2} t_{n,k}^{4(\beta-1)} \var\big(\|\proj_V^\perp\ang\|^2 \un\{\|X\|\in (t_{n,k},d_{n,k}]\}\big)\\
     & \le \frac n{k^2} t_{n,k}^{4(\beta-1)}c_\omega^4 \esp\big( \|X\|^{4(1-\beta)}\un\{\|X\|\in (t_{n,k},d_{n,k}]\}\big)\\
     & \le \frac n{k^2} t_{n,k}^{-\alpha^*}c_\omega^4 \esp\big( \|X\|^{\alpha^*}\un\{\|X\|\in (t_{n,k},d_{n,k}]\}\big).
  \end{align*}
  Similarly as in the proof of Lemma \ref{lem:intconv}, we can bound the expectation using integration by parts and Karamata's theorem:
  \begin{align*}
    \esp\big(& \|X\|^{\alpha^*}\un\{\|X\|\in (t_{n,k},d_{n,k}]\}\big)\\
    & \le \int_0^{d_{n,k}} z^{\alpha^*}\, \PP^{\|X\|}(dz) \\
    & = -d_{n,k}^{\alpha^*} (1-F_{\|X\|}(d_{n,k})) + \alpha^* \int_0^{d_{n,k}} z^{\alpha^*-1}(1-F_{\|X\|}(z))\, dz\\
    & = d_{n,k}^{\alpha^*} (1-F_{\|X\|}(d_{n,k}))\Big( \frac{\alpha^*}{\alpha^*-\alpha}-1 +o(1)\Big)\\
    & \le \frac{2\alpha}{\alpha^*-\alpha}d_{n,k}^{\alpha^*} (1-F_{\|X\|}(d_{n,k}))
  \end{align*}
  for sufficiently large $n$.
  Therefore, by the choice of $d_{n,k}$
  $$ \esp(T_{n,1}^2)=O\Big( \frac n{k^2} (\log k)^{\alpha^*}(1-F_{\|X\|}(d_{n,k}))\Big) =o\Big( \frac{(\log k)^{\alpha^*}}k\Big) = o(1),
  $$
  which implies the convergence in probability of $T_{n,1}$.

  For the second term, we may conclude from \eqref{eq:mombound} and the definition of $t_{n,k}$ that
  \begin{align*}
    \esp(T_{n,2}) & \le \frac nk t_{n,k}^{2(\beta-1)} 2 \esp\big(\|\proj_V^\perp\ang\|^2\un\{\|X\|>d_{n,k}\}\big)\\
    & \le 2 c_\omega^2\frac nk \Big(\frac{d_{n,k}}{t_{n,k}}\Big)^{2(1-\beta)} \esp\|Y_{d_{n,k}}\|^{2(1-\beta)} (1-F_{\|X\|}(d_{n,k}))\\
    & \le \frac{8(1-\beta)c_\omega^2}{\alpha-2(1-\beta)} \point \frac{d_{n,k}^{2(1-\beta)}(1-F_{\|X\|}(d_{n,k}))}{t_{n,k}^{2(1-\beta)}(1-F_{\|X\|}(t_{n,k}))}.
  \end{align*}
  Because $t\mapsto t^{2(1-\beta)}(1-F_{\|X\|}(t))$ is regularly varying with index $2(1-\beta)-\alpha<0$ and $t_{n,k}=o(d_{n,k})$, the right hand side tends to 0. This proves that $T_{n,2}$ converges to 0 in probability, which  concludes the proof.
\end{proof}

The following result is an analog to Lemma \ref{lem:equicont}.
\begin{lemma}  \label{lem:Rnkequicont}
  If $\omega$ satisfies condition \eqref{eq:omegacond}, then for all $\eps>0$ there exists $\delta>0$ such that for sufficiently large $n$
  $$ P\Big\{\sup_{V,W:\rho(V,W)\le\delta} t_{n,k}^{2(\beta-1)}|R_{n,k}(V)-R_{n,k}(W)|>\eps\Big\}<\eps.
  $$
\end{lemma}
\begin{proof}
  First note that in view of \eqref{eq:replacerandthresh}, it suffices to prove the assertion with $R_{n,k}(V)$ replaced by $k^{-1} \sum_{i=1}^n \|\proj_V^\perp\ang_i\|^2 \un\{\|X_i\|>t_{n,k}\}$ and $R_{n,k}(W)$ replaced by the analogous expression.

  Similarly as in the proof of Lemma \ref{lem:equicont}, we have
  \begin{align*}
    \frac 1k& \sum_{i=1}^n \big(\|\proj_V^\perp\ang_i\|^2-\|\proj_W^\perp\ang_i\|^2\big)\un\{\|X_i\|>t_{n,k}\}\\
    & \le \frac 2k \sum_{i=1}^n \|\proj_V^\perp\ang_i-\proj_W^\perp\ang_i\|\point\|\ang_i\|\un\{\|X_i\|>t_{n,k}\}\\
    & \le \frac 2k \rho(V,W) \sum_{i=1}^n \|\ang_i\|^2 \un\{\|X_i\|>t_{n,k}\}\\
    & \le \frac 2k \rho(V,W) c_\omega^2 \sum_{i=1}^n \|X_i\|^{2(1-\beta)} \un\{\|X_i\|>t_{n,k}\}
  \end{align*}
  for $n$ sufficiently large.
  Hence, by Markov's inequality, \eqref{eq:mombound} and the definition of $t_{n,k}$,
  \begin{align*}
    \PP\Big\{ & \sup_{V,W:\rho(V,W)\le\delta}\frac 1k t_{n,k}^{2(\beta-1)}\sum_{i=1}^n \big(\|\proj_V^\perp\ang_i\|^2-\|\proj_W^\perp\ang_i\|^2\big)\un\{\|X_i\|>t_{n,k}\}>\eps\Big\}\\
    & \le \frac{2c_\omega^2\delta n}{\eps k} E\Big(\Big(\frac{\|X\|}{t_{n,k}}\Big)^{2(1-\beta)}
  \un\{\|X\|>t_{n,k}\}\Big)\\
    & \le \frac{8(1-\beta)c_\omega^2\delta }{\eps(\alpha-2(1-\beta))}\\
    & \le \eps
  \end{align*}
  for $\delta:= \eps^2(\alpha-2(1-\beta))/\big(8(1-\beta)c_\omega^2\big)$.
\end{proof}

We are now ready to prove weak consistency of the empirical risk minimizer.
\begin{theorem} \label{th:consistency_hatVn}
  If $\omega$ satisfies condition \eqref{eq:omegacond} and $R_\infty$ has a unique minimizer $V_\infty^*$ in $\VV_p$, then for all minimizers $\hat V_n$ of $R_{n,k}$ in $\VV_p$ one has $\rho(\hat V_n,V_\infty^*)\to 0$ in probability.
\end{theorem}
\begin{proof}
  Let $\tilde R_{n,k} := t_{n,k}^{2(\beta-1)}R_{n,k}$.
  Fix an arbitrary $\eps>0$ and let $\mathcal{M} := \{W\in\VV_p\mid \rho(V_\infty^*,W)\ge\eps/2\}$.
  By the arguments used in the proof of Lemma \ref{lem:equicont}, it is easily seen that $R_\infty$ is (Lipschitz) continuous w.r.t.\ $\rho$. Moreover, $\VV_p$ is compact (see proof of Theorem \ref{th:consistency_Vtstar}), and so is $\mathcal{M}$. Hence, $\eta:=\inf_{W\in\mathcal{M}} R_\infty(W)-R_\infty(V_\infty^*)>0$, since the infimum is attained and $V_\infty^*$ is the unique minimizer of $R_\infty$.

  According to Lemma \ref{lem:Rnkequicont}, there exists $\delta\le\eps/2$ and $n_0$ such that for all $n\ge n_0$ with probability greater than $1-\eps/4$ one has
  $$|\tilde R_{n,k}(V)-\tilde R_{n,k}(W)|\le\eta/4$$
  for all $V,W\in\VV_p$ such that
  $\rho(V,W)\le\delta$. Since $\VV_p$ is compact, there exists a finite cover of $\VV_p$ by open balls with radius $\delta$ and centers $W_1,\ldots, W_m$, say. By Proposition \ref{prop:empriskconsist}, there exists $n_1\ge n_0$ such that with probability greater than $1-\eps/2$
  \begin{align*}
    & |\tilde R_{n,k}(W_j)-R_\infty(W_j)|\le \eta/4, \quad \forall\, 1\le j\le m,\\
    & |\tilde R_{n,k}(V_\infty^*)-R_\infty(V_\infty^*)|\le \eta/4.
  \end{align*}
  Hence, on a set with probability greater than $1-\eps$, there exists $j\in\{1,\ldots,m\}$ such that $\rho(\hat V_n,W_j)<\delta\le\eps/2$ and
  $$ R_\infty(W_j)\le \tilde R_{n,k}(W_j)+\frac\eta 4\le \tilde R_{n,k}(\hat V_n)+\frac\eta 2\le \tilde R_{n,k}(V_\infty^*)+\frac\eta 2\le R_\infty(V_\infty^*)+\frac{3\eta}4.
  $$
  By the definition of $\eta$, this implies $W_j\not\in\mathcal{M}$ and thus
  $$ \rho(\hat V_n,V_\infty)\le \rho(\hat V_n,W_j)+\rho(W_j,V_\infty^*)<\eps. $$
  Since $\eps>0$ was arbitrary, this concludes the proof.
\end{proof}

So far, we have proved weak consistency of both the standardized empirical risk and the empirical risk minimizer under mild assumptions on the scaling function $\omega$. However, the rates of convergence may be arbitrarily slow. As condition \eqref{eq:omegacond} does not guarantee any finite moments of $\ang$ of order greater than 1, it will not suffice to establish useful risk bounds. In the next section, we therefore analyze the recovery risk under the stronger assumption that $\ang$ is bounded.

\section{Uniform risk bounds}
 \label{sec:recover}

Since a minimizer $\hat V_t$ of the empirical risk $\hat R_t$ (or $\hat V_n$ of $R_{n,k}$) differs from the minimizer of the true risk $R_t$, usually the so-called excess risk $R_t(\hat V_t)-\inf_{V\in\VV_p} R_t(V)$ will be strictly positive. We follow the common approach in the theory of risk minimization to bound the excess risk by deriving uniform bounds on $|\hat R_t(V)-R_t(V)|$ which hold with high probability for a fixed sample size $n$. If these uniform bounds can be calculated from the observed data, they may also be used to construct confidence intervals for the reconstruction error $R_t(\hat V_t)$ resp.\ $R_{t_{n,k}}(\hat V_n)$.

Since tight concentration inequalities are available only for subgaussian distributions, in this section we will assume that the scaling function $\omega$ satisfies the following condition:
\begin{equation}\label{eq:omegacond2}
  \omega(x)\le \frac 1{\|x\|},\quad \forall\, x\in\R^d,
\end{equation}
so that $\|\angf(x)\|\le 1$ for all $x\in\R^d$. Moreover, we suppose that the \cdf of $\|X\|$ is continuous in the tail to avoid technicalities. Then we may assume w.l.o.g.\ that there are no ties and thus exactly $k$ observations with norm larger than $\hat t_{n,k}$.

For classical PCA  (and a kernel version thereof), \cite{shawe2005eigenspectrum} established uniform risk bounds for bounded random vectors $Z_i$, which were improved by the following result by \cite{blanchard2007statistical}. Assume $\|Z_i\|\le 1$, denote the empirical matrix of second (mixed) moments by $\hat\Sigma_n$ and the Hilbert-Schmidt norm on the space of matrices by $\|\point\|_{HS}$. Then, with probability greater than $1-\delta$
\begin{align*}
\Big| \frac 1n\sum_{i=1}^n \|\proj_V^\perp Z_i\|^2- \esp\|\proj_V^\perp Z\|^2\Big| &\le \bigg[\frac p{n-1} \Big(\frac 1n\sum_{i=1}^n \|Z_i\|^4-\|\hat\Sigma_n\|_{HS}\Big)\bigg]^{1/2} \\
  & \quad + \Big(\frac{\log(3/\delta)}{2n}\Big)^{1/2} + \Big(\frac{p^2 \log(3/\delta)}{n^3}\Big)^{1/4}
\end{align*}
for all $V\in\VV_p$.
One may try to derive uniform risk bounds in our extreme value setting by applying this result to the random variables $Z_i=\ang_{i,t} = \ang_i\un\{\|X_i\|>t\}$, so that the left hand side is approximately equal to $\pi_t| \hat R_t(V)-R_t(V)|$  with
$$ \pi_t := \PP\{\|X\|>t\} $$
if one ignores the difference between $N_t$ and its expectation $n\pi_t$.
In the case $\pi_t=o(n^{-1/2})$, however, the above upper bound will not even converge to 0 when it is divided by $\pi_t$ because of the second term. Hence this direct approach does not give meaningful bounds for $|\hat R_t(V)-R_t(V)|$.

The reason for this inconsistency is that, unlike in the classical setting, most of the random variables $Z_i$ will vanish as $t$ increases, and the concentration inequalities used in the proofs of the aforementioned bounds are too crude in such a situation. However, we will take up ideas used by \cite{blanchard2007statistical}, with appropriate modifications, to derive much tighter uniform bounds on $|R_{n,k}(V)-R_{t_{n,k}}(V)|$. Furthermore, we will derive uniform bounds on $|\hat R_t(V)-R_t(V)|$ which hold conditionally on $N_t=\ell$ and depend only on  the data. These can then be used to construct confidence bands for $R_t(V)$.

Before we establish these bounds, we first recall some well-known facts about Hilbert spaces specialized to the present setting, and introduce some notation. Let $(e_i)_{1\le i\le d}$ be an arbitrary orthonormal basis of $\R^d$ and denote by $\langle\point,\point\rangle$ the usual inner product on $\R^d$.
The space of linear operators from $\R^d$ to $\R^d$ (i.e., $d\times d$-matrices) equipped with the inner product $\langle A,B\rangle_{HS} := \sum_{i=1}^d\langle A e_i,Be_i\rangle$ is a Hilbert space. The corresponding Hilbert Schmidt norm can be expressed as $\|A\|_{HS}=\big(\sum_{i=1}^d \|Ae_i\|^2\big)^{1/2}=\big(\tr(AA^\top)\big)^{1/2}$ with $\tr$ denoting the trace operator. If, for any subspace $W$ of $\R^d$, one chooses the first $\dim W$ vectors $e_i$ to form an  orthonormal basis of $W$, then one sees that
\begin{equation} \label{eq:projnorm}
\|\proj_W\|_{HS}=\sqrt{\dim W}.
\end{equation}
 Moreover, direct calculations show that
\begin{equation}  \label{eq:scalarprod}
\langle Ay,x\rangle  =\langle A,xy^\top\rangle_{HS}.
\end{equation}
Finally, for independent centered random matrices $A_i$, $1\le i\le n$, one has
\begin{equation} \label{eq:indepsum}
  \esp \Big\|\sum_{i=1}^n A_i  \Big\|_{HS}^2 =   \sum_{i=1}^n\esp\|A_i\|_{HS}^2.
\end{equation}

 If, for the time being, one neglects the difference between the empirical and the true $(1-k/n)$-quantile of $\|X\|$, then $R_{n,k}(V)$ can be approximated by $\bar R_{t_{n,k}}(V)$ where
 \begin{equation} \label{eq:Rttildedef}
 \bar R_t(V) := \frac 1{n\pi_t} \sum_{i=1}^n \|\proj_V^\perp\Theta_{i,t}\|^2.
 \end{equation}
Denote the empirical distribution of the observed random vectors $X_i$, $1\le i\le n$,  by $P_n$.
For any threshold $t>0$, the maximal difference between the approximate empirical risk $\bar R_t(V)$ and the true risk $R_t(V)$ can be rewritten as
\begin{align*}
 \sup_{V\in\VV_p} |\bar R_t(V)-R_t(V)| & = \sup_{V\in\VV_p} \frac 1{\pi_t}\Big|\frac 1n \sum_{i=1}^n \|\proj_V^\perp\ang_{i,t}\|^2- \esp\|\proj_V^\perp\ang_t\|^2\Big|\\
  & = \frac 1{\pi_t} \sup_{V\in\VV_p} \big| (P_n-P)(\|\proj_V^\perp\angf_t\|^2)\big|\\
  & = \frac 1{\pi_t} \max(\varphi^+(X_1,\ldots,X_n),\varphi^-(X_1,\ldots,X_n))
\end{align*}
with
$$ \varphi^\pm_t(x_1,\ldots,x_n) := \sup_{V\in\VV_p} \pm\Big(\frac 1n\sum_{i=1}^n\|\proj_V^\perp\angf_t(x_i)\|^2-P\|\proj_V^\perp\angf_t\|^2\Big). $$
For brevity's sake, in what follows  we use the notation $x_{i:j}:=(x_i,\ldots,x_j)$ for a subvector of $(x_1,\ldots,x_n)$.

In order to derive uniform bounds on the difference between the empirical and the true risk, we first establish concentration inequalities for $\varphi_t^\pm(X_1,\ldots,X_n)$ using a version of the bounded difference inequality by \cite[Theorem 3.8]{mcdiarmid1998concentration}, which we recall for convenience.
\begin{theorem}
\label{colt:thm-berstein}
Let $ X_{1:n} = ( X_1,\ldots, X_n) $ be an \iid sample  taking its values in some space $ E$ and $\varphi:{ E}^n\to\rset$ be any measurable function.
Consider the positive deviation functions, defined for $1\le m\le n$ and for ${ x}_{1:m}\in { E}^m$,
\[
h_m( x_{1:m}) = \esp\big(\varphi( x_{1:m},  X_{m+1:n})  -
\varphi( x_{1:{m-1}},  X_{m:n} )\big).
\]
Denote their maximum by
\begin{equation}
  \label{eq:maxdev}
  \textrm{maxdev}^+ = \max_{1\le m\le n} \sup_{x_{1:m}\in E^m} h_m(x_{1:m}),
\end{equation}
and the maximal summed variance by
\begin{equation}
  \label{eq:maxSvar}
  \hat v = \sup_{x_{1:n}\in E^n}\sum_{m=1}^n \var h_m( x_{1:m-1}, X_m).
\end{equation}
If both $\textrm{maxdev}^+$ and $\hat v$ are finite, then for all $u \ge 0$,
\begin{align*}
\mathbb{P} \big\{ \varphi( X_{1:n}) - \esp\varphi( X_{1:n})\ge u  \big\} \le \exp{\left(- \frac{u^2}{ 2(\hat v +  \textrm{maxdev}^+ u/3 )} \right)}.
\end{align*}
\end{theorem}
\begin{lemma} \label{lem:concentrationmaxdiff}
  For all $u>0$,
  $$ \PP\big\{\varphi_t^\pm(X_1,\ldots,X_n)\ge \esp  \varphi_t^\pm(X_1,\ldots,X_n)+u\big\}\le \exp\Big(-\frac{nu^2}{2(\pi_t(1+\pi_t)+u/3)}\Big). $$
\end{lemma}
\begin{proof}
   The assertion follows immediately from Theorem \ref{colt:thm-berstein} applied to $\varphi_t^\pm$ and the following bounds:
   \begin{align*}
     h_m( x_{1:m}) & = \esp\big(\varphi^\pm_t(x_{1:m},X_{m+1:n})-\varphi^\pm_t(x_{1:m-1},X_{m:n})\big)\\
      & \le \frac 1n \esp\Big( \sup_{V\in\VV_p} \Big| \|\proj_V^\perp \angf_t(X_m)\|^2-\|\proj_V^\perp\angf_t(x_m) \|^2\Big|\Big)\\
      & \le \frac 1n \esp\big(\un\{\|X_m\|>t\| \text{ or } \|x_m\|>t\}\big)\\
      & = \frac 1n(\pi_t\vee \un\{\|x_m\|>t\})\\
      & \le \frac 1n
   \end{align*}
   and
   \begin{align*}
    \sum_{m=1}^n \var h_m(x_{1:m-1},X_m) & \le \sum_{m=1}^n \esp h_m^2(x_{1:m-1},X_m)\\
    & \le \frac 1n \esp(\pi_t\vee \un\{\|X\|>t\})^2 \\
    & = \frac{\pi_t^2(1-\pi_t)+\pi_t }n\\
    & \le \frac{\pi_t(1+\pi_t)}n.
   \end{align*}
\end{proof}
The expectation $\esp  \varphi_t^\pm(X_1,\ldots,X_n)$ can be bounded using arguments from \cite{blanchard2007statistical}.
\begin{lemma} \label{lem:expectationmaxdiff}
  $$ \esp  \varphi_t^\pm(X_1,\ldots,X_n)\le \Big[ \frac{p\wedge(d-p)}n \pi_t\big(\esp_t\|\ang\|^4-\pi_t \tr(\Sigma_t^2)\big)\Big]^{1/2} $$
  with $\Sigma_t := \esp_t\ang\ang^\top$.
\end{lemma}
\begin{proof}
  Since, by \eqref{eq:scalarprod}, $\|\proj_Wx\|^2 = \langle \proj_W x,x\rangle = \langle \proj_W,xx^\top\rangle_{HS}$ for any linear subspace $W$ and any $x\in\R^d$, using the bilinearity of the inner product and the Cauchy-Schwarz inequality in the Hilbert-Schmidt space, we obtain
  $$ \pm(P_n-P)(\|\proj_V^\perp\angf_t\|^2) = \Big\langle \proj_V^\perp, \pm(P_n-P)(\angf_t\angf_t^\top)\Big\rangle_{HS} \le \|\proj_V^\perp\|_{HS} \|(P_n-P)(\angf_t\angf_t^\top)\|_{HS}.
  $$
  Using \eqref{eq:projnorm} and taking the supremum over all $V\in\VV_p$ and the expectation, one arrives at
  \begin{equation} \label{eq:expectvarbound1}
   \esp\varphi_t^\pm(X_1,\ldots,X_n) \le \sqrt{d-p}\esp\|(P_n-P)(\angf_t\angf_t^\top)\|_{HS}.
  \end{equation}
  One the other hand, by first rewriting $\|\proj_V^\perp \angf_t\|^2= \|\angf_t\|^2-\|\proj_V \angf_t\|^2$, analogously one obtains
  \begin{align}
   \esp\varphi_t^\pm (X_1,\ldots,X_n)& = \esp\Big(\sup_{V\in\VV_p}\pm(P_n-P)(\|\proj_V^\perp\angf_t\|^2)\Big)  \nonumber\\
      &= \esp\big((P_n-P)\|\angf_t\|^2\big) + \esp\Big(\sup_{V\in\VV_p}\mp (P_n-P)(\|\proj_V\angf_t\|^2)\Big)  \nonumber\\
       & \le  0 +  \sup_{V\in\VV_p}\|\proj_V\|_{HS} \esp\|(P_n-P)(\angf_t\angf_t^\top)\|_{HS}  \nonumber\\
       & \le \sqrt{p} \esp\|(P_n-P)(\angf_t\angf_t^\top)\|_{HS}.  \label{eq:expectvarbound2}
   \end{align}
   Now, by the Cauchy-Schwarz inequality and \eqref{eq:indepsum},
   \begin{align*}
    \esp\|(P_n-P)(\angf_t\angf_t^\top)\|_{HS} & \le \big( \esp\|(P_n-P)(\angf_t\angf_t^\top)\|^2_{HS}\big)^{1/2}\\
     & = \Big( \esp\Big\|\frac 1n \sum_{i=1}^n \big(\ang_{i,t}\ang_{i,t}^\top-\esp(\ang_{i,t}\ang_{i,t}^\top)\big)\Big\|^2_{HS}\Big)^{1/2}\\
     & = \Big(\frac 1n \esp\|\ang_t\ang_t^\top- \esp\ang_t\ang_t^\top\|^2_{HS}\Big)^{1/2}.
   \end{align*}
   Combining this with \eqref{eq:expectvarbound1} and \eqref{eq:expectvarbound2}, we arrive at
   $$ \esp  \varphi_t^\pm(X_1,\ldots,X_n)\le \Big[ \frac{p\wedge(d-p)}n \esp\|\ang_t\ang_t^\top- \esp\ang_t\ang_t^\top\|^2_{HS}\Big]^{1/2}. $$

   It remains to show that $\esp\|\ang_t\ang_t^\top- \esp\ang_t\ang_t^\top\|^2_{HS}=\pi_t\big(\esp_t\|\ang\|^4-\pi_t \tr(\Sigma_t^2)\big)$. From the representation of the Hilbert Schmidt norm by the trace operator and the linearity of the latter, one may conclude by direct calculations that
   \begin{align*}
      \esp\|\ang_t\ang_t^\top- \esp\ang_t\ang_t^\top\|^2_{HS} & = \tr\big(\esp (\ang_t\ang_t^\top-\esp\ang_t\ang_t^\top)^2\big)\\
      & = \tr\big( \esp(\ang_t\ang_t^\top)^2\big)- \tr\big((\esp\ang_t\ang_t^\top)^2\big)\\
      & = \pi_t \tr\big( \esp_t(\ang\ang^\top)^2\big)- \tr\big((\pi_t\esp_t\ang\ang^\top)^2\big)\\
      & = \pi_t \esp_t \tr\big((\ang\ang^\top)^2\big)-\pi_t^2 \tr(\Sigma_t^2).
   \end{align*}
   Hence the assertion follows from
   $$ \tr\big((\ang\ang^\top)^2\big) = \sum_{j=1}^d \|\ang\ang^\top e_j\|^2 = \sum_{j=1}^d \sum_{l=1}^d \big(\ang^{(l)}\ang^{(j)}\big)^2 = \|\ang\|^4
   $$
   with $e_j$ denoting the $j$th unit vector.
\end{proof}

Now we are ready to state our first uniform risk bound. Recall that $\Sigma_t:=\esp_t(\ang\ang^\top)$.
\begin{theorem} \label{th:unifboundRnk}
  For all $u,v>0$ one has
  \begin{align}
    \PP\Big\{\sup_{V\in\VV_p} & |R_{n,k}(V)-R_{t_{n,k}}(V)|\ge \Big[\frac{p\wedge (d-p)}k S_{t_{n,k}}\Big]^{1/2}+u+v\Big\} \nonumber\\
     & \le 2\exp\Big(-\frac{ku^2}{2(1+k/n+u/3)}\Big) + 2\exp\Big( -\frac{kv^2}{2(1+v/3)}\Big) \label{eq:unifboundRnk1}
  \end{align}
  with $S_t:= \esp_t\|\ang\|^4-\pi_t\tr(\Sigma_t^2)$.

  In particular, with probability greater than or equal to $1-\delta$
  \begin{align}
   \sup_{V\in\VV_p} & |R_{n,k}(V)-R_{t_{n,k}}(V)|  \nonumber\\
    &\le \Big[\frac{p\wedge (d-p)}k S_{t_{n,k}}\Big]^{1/2}
     + \frac{2\log(4/\delta)}{3k} + \Big[ \Big(\frac{\log(4/\delta)}{3k}\Big)^2+ \frac 2k (1+k/n)\log(4/\delta)\Big]^{1/2} \nonumber\\
      & \quad + \Big[ \Big(\frac{\log(4/\delta)}{3k}\Big)^2+ \frac 2k \log(4/\delta)\Big]^{1/2} \nonumber\\
    & \le \Big[\frac{p\wedge (d-p)}k S_{t_{n,k}}\Big]^{1/2}
      + \Big[  \frac 8k (1+k/n)\log(4/\delta)\Big]^{1/2}+ \frac{4\log(4/\delta)}{3k}. \label{eq:unifboundRnk2}
    \end{align}
\end{theorem}
Note that \eqref{eq:unifboundRnk2} also implies an upper bound on the excess risk:
\begin{align*}
  R_{t_{n,k}}(\hat V_n)-\inf_{V\in\VV_p} R_{t_{n,k}}(V) & \le R_{n,k}(\hat V_n) - R_{t_{n,k}}(V_{t_{n,k}}^*) + \sup_{V\in\VV_p} |R_{n,k}(V)-R_{t_{n,k}}(V)| \\
  & \le  R_{n,k}(V_{t_{n,k}}^*) - R_{t_{n,k}}(V_{t_{n,k}}^*) + \sup_{V\in\VV_p} |R_{n,k}(V)-R_{t_{n,k}}(V)| \\
  & \le 2\sup_{V\in\VV_p} |R_{n,k}(V)-R_{t_{n,k}}(V)|.
\end{align*}
\begin{proof}
  With $   \bar R_t(V)$ defined in \eqref{eq:Rttildedef},
  we have
  $$ \sup_{V\in\VV_p} |R_{n,k}(V)-R_{t_{n,k}}(V)| \le  \sup_{V\in\VV_p} |R_{n,k}(V)-\bar R_{t_{n,k}}(V)| +  \sup_{V\in\VV_p} |\bar R_{t_{n,k}}(V)-R_{t_{n,k}}(V)|.
  $$
  By similar arguments as in the proof of Proposition \ref{prop:empriskconsist}, we see that, for all $V\in\VV_p$,
  \begin{align*}
   |R_{n,k}(V)-\bar R_{t_{n,k}}(V)| & = \frac 1k \Big|\sum_{i=1}^n \|\proj_V^\perp\ang_i\|^2(\un\{\|X_i\|>\hat t_{n,k}\}-\un\{\|X_i\|>t_{n,k}\})\Big|\\
   & \le \frac 1k \Big|\sum_{i=1}^n \un\{\|X_i\|>t_{n,k}\}-k\Big|.
  \end{align*}
  By Bernstein's inequality, it follows that
  \begin{align*}
   \PP\Big\{\sup_{V\in\VV_p} |R_{n,k}(V)-\bar R_{t_{n,k}}(V)|\ge v\Big\}
   &\le 2\exp\Big( -\frac{(kv)^2}{2(k(1-k/n)+kv/3)}\Big) \\
   & \le 2\exp\Big( -\frac{kv^2}{2(1+v/3)}\Big).
  \end{align*}
  For the second term
  $ \sup_{V\in\VV_p} |\bar R_{t_{n,k}}(V)-R_{t_{n,k}}(V)|= \frac nk \max(\varphi^+_{t_{n,k}}(X_{1:n}),\varphi^-_{t_{n,k}}(X_{1:n}))$, Lemma \ref{lem:concentrationmaxdiff}, Lemma \ref{lem:expectationmaxdiff}  and $\pi_{t_{n,k}}=k/n$ immediately yield
  \begin{align*}
   \PP\Big\{\sup_{V\in\VV_p} &|\bar R_{t_{n,k}}(V)-R_{t_{n,k}}(V)|\ge \Big[\frac{p\wedge (d-p)}k S_{t_{n,k}}\Big]^{1/2}+u\Big\} \\
   & \le 2 \exp\Big(-\frac{n(uk/n)^2}{2(k/n(1+k/n)+uk/(3n))}\Big)\\
   & = 2 \exp\Big(-\frac{ku^2}{2(1+k/n+u/3)}\Big),
  \end{align*}
  which concludes the proof of the first assertion.

  Check that for
  \begin{align*}
    u& := \frac{\log(4/\delta)}{3k} + \Big[ \Big(\frac{\log(4/\delta)}{3k}\Big)^2+ \frac 2k (1+k/n)\log(4/\delta)\Big]^{1/2}\\
    v & := \frac{\log(4/\delta)}{3k} + \Big[ \Big(\frac{\log(4/\delta)}{3k}\Big)^2+ \frac 2k \log(4/\delta)\Big]^{1/2}
  \end{align*}
  both exponential expressions on the right hand side of \eqref{eq:unifboundRnk1} equal $\delta/4$, and so the upper bound equals $\delta$. Hence the remaining assertions follow from $\sqrt{a+b}\le \sqrt{a}+\sqrt{b}$.
\end{proof}

\begin{remark}
  In the case $\omega(x)=\|x\|$, the upper bound in \eqref{eq:unifboundRnk2} simplifies to
  $$ \Big[\frac{p\wedge (d-p)}k \big(1-(k/n)\tr(\Sigma_{t_{n,k}}^2)\big)\Big]^{1/2}
      + \Big[  \frac 8k (1+k/n)\log(4/\delta)\Big]^{1/2}+ \frac{4\log(4/\delta)}{3k}.\hspace*{\fill}\qed
  $$
\end{remark}

Note that the upper bound in Theorem  \ref{th:unifboundRnk} cannot be calculated from the data and can thus not directly be used to construct confidence intervals for the true reconstruction error $R_{t_{n,k}}(\hat V_n)$ or the minimal reconstruction error $\inf_{V\in\VV_p} R_{t_{n,k}}(V)$. Next, we derive data-dependent bounds directly from (a minor improvement of) the bound established by \cite{blanchard2007statistical}. However, this result will be applied to the conditional distribution of $\ang$ given $\|X\|>t$ and the resulting bound is to be interpreted conditional on the number $N_t$ of exceedances over the chosen threshold $t$.

\begin{theorem} \label{th:condapproach}
   For all $\ell>1, u,v>0$,
   \begin{align*}
     \PP\Big( \sup_{V\in\VV_p} & |\hat R_t(V)-R_t(V)|\ge \Big[(p\wedge (d-p))\Big(\frac{\tilde S_t}{\ell-1}+\frac{v}{\ell}\Big)\Big]^{1/2}+u \,\Big|\, N_t=\ell\Big) \\
      & \le 2\exp\big(-2\ell u^2)+\exp\big(-\floor{\ell/2} v^2/2\big)
   \end{align*}
   with $\tilde S_t := N_t^{-1}\sum_{i=1}^n \|\ang_{i,t}\|^4-\tr\Big((N_t^{-1}\sum_{i=1}^n\ang_{i,t}\ang_{i,t}^\top)^2\Big)$ and $\floor{x}:=\max\{k\in\mathbb{Z}\mid k\le x\}$.

   If, for all $\ell>1$, constants $u_\ell,v_\ell>0$ are chosen such that $2\exp\big(-2\ell u_\ell^2)+\exp\big(-\floor{\ell/2} v_\ell^2/2\big)=1-\alpha$, then
   $I_\ell(V) := \big[\hat R_t(V)-B_{t,\ell}, \hat R_t(V)+B_{t,\ell}\big]\cap [0,\infty)$ with
   $$ B_{t,\ell} :=  \Big[(p\wedge (d-p))\Big(\frac{\tilde S_t}{\ell-1}+\frac{v_\ell}{\ell}\Big)\Big]^{1/2}+u_\ell
   $$
   defines a uniform level $\alpha$ confidence band for $R_t(V)$, $V\in\VV_p$, conditionally on $N_t=\ell$. If one defines $I_0(V)=I_1(V)=[0,\infty)$, then $I_{N_t}(V)$ defines a uniform level $\alpha$ confidence band for $R_t(V)$, $V\in\VV_p$ (unconditionally).
\end{theorem}
\begin{proof}
   Define \iid  random vectors $Z_i$  whose distribution equals the conditional distribution of $\ang$ given $\|X\|>t$.
   Recall that $\ang_{(i)}:=\angf(X_{(i)})$ where $X_{(i)}$ is the vector $X_j$ with the $i$th largest norm among $X_1,\ldots,X_n$. Then, conditionally on $N_t=\ell$, the joint distribution of the empirical risk $\hat R_t(V)$ and $\ang_{(1)},\ldots\ang_{(\ell)}$ equals the joint distribution of $\ell^{-1} \sum_{i=1}^\ell \|\proj_V^\perp Z_i\|^2$ and the order statistics of $Z_1,\ldots,Z_\ell$. Therefore, the proof of Theorem 3.1 of \cite{blanchard2007statistical} (with $M=1$ and $L=2$) combined with arguments given in the proof of Lemma \ref{lem:expectationmaxdiff} show that
   \begin{align}
   \PP\Big( \sup_{V\in\VV_p} & |\hat R_t(V)-R_t(V)|\ge \Big[\frac{p\wedge (d-p)}{2\ell}\Big(\frac 1{\ell(\ell-1)}\sum_{i,j=1}^\ell \|\ang_{(i)}\ang_{(i)}^\top- \ang_{(j)}\ang_{(j)}^\top\|_{HS}^2+2v\Big)\Big]^{1/2} \nonumber\\
     & { }  +u \,\Big|\, N_t=\ell\Big)  \le 2\exp\big(-2\ell u^2)+\exp\big(-\floor{\ell/2} v^2/2\big).  \label{eq:Blanchardbound}
    \end{align}
   Since the proof of Theorem 3.1 of \cite{blanchard2007statistical}  is quite tersely formulated in a more abstract setting and contains a minor inaccuracy, for convenience we give more details of the proof of \eqref{eq:Blanchardbound} in the Appendix.

   In the same way  as in the proof of Lemma \ref{lem:expectationmaxdiff}, the first assertion thus follows from
   \begin{align*}
     \sum_{i,j=1}^\ell \|\ang_{(i)}\ang_{(i)}^\top- \ang_{(j)}\ang_{(j)}^\top\|_{HS}^2
     & = 2\ell \sum_{i=1}^\ell \|\ang_{(i)}\ang_{(i)}^\top\|^2_{HS}-2 \sum_{i,j=1}^\ell \langle \ang_{(i)}\ang_{(i)}^\top,\ang_{(j)}\ang_{(j)}^\top\rangle_{HS} \\
     & = 2\ell \sum_{i=1}^\ell \|\ang_{(i)}\|^4 - 2\Big\|\sum_{i=1}^\ell \ang_{(i)}\ang_{(i)}^\top\Big\|^2_{HS}\\
     & = 2\ell^2 \bigg(\frac 1\ell \sum_{i=1}^\ell \|\ang_{(i)}\|^4 - \tr\Big(\Big(\frac 1\ell\sum_{i=1}^\ell \ang_{(i)}\ang_{(i)}^\top\Big)^2\Big)\bigg)\\
     & = 2\ell^2 \bigg(\frac 1\ell \sum_{i=1}^n \|\ang_{i,t}\|^4 - \tr\Big(\Big(\frac 1\ell\sum_{i=1}^n \ang_{i,t}\ang_{i,t}^\top\Big)^2\Big)\bigg)
   \end{align*}
   where in the last step we have used that, on $\{N_t=\ell\}$, the set of  non-vanishing vectors $\ang_{i,t}$ equals the set of non-vanishing random vectors $\ang_{(i)}$.

   The remaining assertions are now obvious.
\end{proof}

\begin{remark} \label{rem:shortconfbands}
   In the statement about the confidence bands one may replace $B_{t,\ell}$ with
    $$\tilde B_{t,\ell} := \Big[(p\wedge (d-p))\frac{\tilde S_t}{\ell-1}\Big]^{1/2}+\Big[(p\wedge (d-p))\frac{v_\ell}{\ell}\Big]^{1/2}+u_\ell.
   $$
   This half width of a confidence band is more suitable for (numerical) minimization (as a function of $u_\ell$ and $v_\ell$) under the constraint $2\exp\big(-2\ell u_\ell^2)+\exp(-\floor{\ell/2} v_\ell^2/2)=1-\alpha$.\qed
\end{remark}

\begin{remark}
  The (modified) proof of Theorem 3.1 of \cite{blanchard2007statistical} also shows that
   \begin{align*}
    \PP\Big( \sup_{V\in\VV_p}  |\hat R_t(V)-R_t(V)|\ge \Big[\frac{p\wedge (d-p)}{\ell} S_t^*\Big]^{1/2} +u  \,\Big|\, N_t=\ell\Big)  \le 2\exp\big(-2\ell u^2)
   \end{align*}
   with $S_t^* := \esp_t\|\ang\|^4-\tr(\Sigma_t^2)$.
   Observe that $\bar R_t(V)=N_t\hat R_t(V)/(n\pi_t)$.
   On the set $M_t(v):=\{|N_t-n \pi_t|\le n\pi_t v\}$, one thus has
   \begin{align*}
     \sup_{V\in\VV_p} |\bar R_t(V)-R_t(V)| \le \frac{N_t}{n\pi_t} \sup_{V\in\VV_p}|\hat R_t(V)-R_t(V)|+ v,
   \end{align*}
   since $\hat R_t(V)\le 1$.
   Moreover, for $t=t_{n,k}$, it has been shown in the proof of Theorem \ref{th:unifboundRnk} that $\sup_{V\in\VV_p}|R_{n,k}(V)-\bar R_{t_{n,k}}(V)|\le v$ on the set $M_{t_{n,k}}=\{N_{t_{n,k}}\in[k(1-v),k(1+v)]\}$ and that $\PP(M_{t_{n,k}}^c)\le 2\exp\big(-kv^2/(2(1+v/3))\big)$. Hence,
   \begin{align*}
     \PP&\Big\{\sup_{V\in\VV_p}  |R_{n,k}(V)-R_{t_{n,k}}(V)|\ge \Big[ (1+v)\frac{p\wedge (d-p)}kS_{t_{n,k}}^*\Big]^{1/2}+u+2v\Big\}\\
     & \le \PP\Big(M_{t_{n,k}}\cap\Big\{\sup_{V\in\VV_p}  |\bar R_{t_{n,k}}(V)-R_{t_{n,k}}(V)|\ge \Big[ (1+v)\frac{p\wedge (d-p)}kS_{t_{n,k}}^*\Big]^{1/2}+u+v\Big\}\Big)\\
     & \hspace*{11cm} +\PP(M_{t_{n,k}}^c)\\
     & \le \PP\Big(M_{t_{n,k}}\cap\Big\{\sup_{V\in\VV_p}  |\hat R_{t_{n,k}}(V)-R_{t_{n,k}}(V)|\ge \Big[\frac{p\wedge (d-p)}{N_{t_{n,k}}}S_{t_{n,k}}^*\Big]^{1/2}+\frac{ku}{N_{t_{n,k}}}\Big\}\Big)\\
     & \hspace*{11cm} +\PP(M_{t_{n,k}}^c)\\
     & \le \sum_{\ell=\ceil{k(1-v)}}^{\floor{k(1+v)}} 2\exp\big(-2\ell (ku/\ell)^2\big)\PP\{N_{t_{n,k}}=\ell\} +\PP(M_{t_{n,k}}^c)\\
     & \le 2\exp\Big(-\frac{2ku^2}{1+v}\Big) + 2\exp\Big(-\frac{kv^2}{2(1+v/3)}\Big).
   \end{align*}
   A comparison with Theorem \ref{th:unifboundRnk} reveals that the new bound may be tighter if $S_{t_{n,k}}^*$ is substantially smaller than $S_{t_{n,k}}$. This will be the case if $k/n$ is small and $\tr\big((\esp_t\ang\ang^\top)^2\big)$ is not much smaller than $\esp_t\|\ang\|^4$. \qed
\end{remark}

So far, we have compared empirical risks with the true risk $R_t$ for finite thresholds $t$. A comparison with the limit risk $R_\infty$ would require second order refinements of our basic assumption \eqref{eq:rvalpha}. Let $\Sigma_t:=\esp_t \ang\ang^\top=P_t(\angf\angf^\top)$ and $\Sigma_\infty=P_\infty(\angf\angf^\top)$. Denote the eigenvalues of $\Sigma_t-\Sigma_\infty$ by $\lambda_{t,1}^{\Delta}\ge\lambda_{t,2}^{\Delta}\ge\ldots\ge\lambda_{t,n}^{\Delta}$. Then  standard calculations from classical PCA show that
$$ \sup_{V\in\VV_p} R_t(V)-R_\infty(V)=\sup_U \tr(U^\top(\Sigma_t-\Sigma_\infty)U) = \sum_{i=1}^{d-p} \lambda^\Delta_{t,i}
$$
where the second supremum is taken over all $(d\times(d-p))$-matrices with orthogonal columns. Likewise, $\sup_{V\in\VV_p} R_\infty(V)-R_t(V)= -\sum_{i=1}^{d-p} \lambda^\Delta_{t,d+1-i}$ and hence
$$ \sup_{V\in\VV_p} |R_t(V)-R_\infty(V)|\le \max\Big(\Big|\sum_{i=1}^{d-p}\lambda^\Delta_{t,i}\Big|,\Big|\sum_{i=p+1}^d|\lambda^\Delta_{t,i}\Big|\Big)
= \max\Big(\Big|\sum_{i=1}^p\lambda^\Delta_{t,i}\Big|,\Big|\sum_{i=d-p+1}^d\lambda^\Delta_{t,i}\Big|\Big).
$$
Therefore, bounds on the difference between empirical risks and the limit risk require additional assumptions on the spectrum of the difference $\Sigma_t-\Sigma_\infty$ between the matrix of second moments for the re-scaled exceedances over the threshold $t$ and the corresponding matrix in the limit model.

If one merely wants to compare the minimum risk for finite thresholds with the minimum limit risk, which equal  the sums of $d-p$ smallest eigenvalues of $\Sigma_t$ resp.\ $\Sigma_\infty$, then somewhat weaker assumptions on the convergence of the spectrum of $\Sigma_t$ and $\Sigma_\infty$ are needed. In particular, under Hypothesis 1, $\inf_{V\in\VV_p} R_t(V) - \inf_{V\in\VV_p} R_\infty(V)$ equals the sum of the smallest $d-p$ eigenvalues of $\Sigma_t$.

\section{Simulation study}

We investigate the performance of our PCA procedure. In particular, we examine how the standard non-parametric estimator of the spectral measure (defined via \eqref{eq:spectraldef}) based on the $k$ largest observations
$$ \hat H_{n,k} := \frac 1k \sum_{i=1}^n \delta_{\angf_{t_{n,k}}(X_i)} $$
(with $\angf(x)=x/\|x\|$) is influenced if the data is first projected onto a lower dimensional subspace using PCA:
$$ \hat H_{n,k}^{PCA} := \frac 1k \sum_{i=1}^n \delta_{\angf_{t_{n,k}}(\proj_V^\perp X_i)}. $$
Here, $\delta_y$ is the Dirac measure with point mass at $y$ and $V$ denotes the subspace picked by PCA based on the same number $k$ of largest observations. It will turn out that sometimes it is advisable to use a smaller number $\tilde k$ for the PCA procedure; the resulting estimator of the spectral measure will be denoted by $\hat H_{n,k,\tilde k}^{PCA}$.

To measure the performance of the spectral estimators, we consider the resulting estimators of the following probabilities in the limit model, that can be expressed in terms of the spectral measure:
\begin{itemize}
   \item[(i)] $\lim_{u\to\infty} \PP(p^{-1}\sum_{1\le j\le p} X^j/\|X\|>t_{(i)} \mid \|X\|>u)=H\{x\mid p^{-1} \sum_{j=1}^p x^j>t_{(i)}\}$ for some $t_{(i)}\in (0,p^{-1/2})$
   \item[(ii)] $\lim_{u\to\infty} \PP(\min_{1\le j\le p}X^j>u,\max_{p+1\le j\le d} X^j\le u \mid \|X\|>u)=$\\ $ \int \big((\min_{1\le j\le p}x^j)^\alpha-(\max_{p+1\le j\le d}x^j)^\alpha\big)^+\, H(dx)$
    \item[(iii)] $ \lim_{u\to\infty} \PP(X^1>u\mid \max_{1\le j\le d}X^j>u) = \int (x^1)^\alpha\, H(dx)/\int (\max_{1\le j\le p}x^j)^\alpha\, H(dx)$
   \item[(iv)] $ \lim_{u\to\infty} \PP(\min_{1\le j\le d}X^j>u\mid \|X\|>u) = \int (\min_{1\le j\le d}x^j)^\alpha\, H(dx)$

 \end{itemize}
The first probability is related to the cdf of the mean contribution  of the first $p$ coordinates to the norm of the random vector, thus quantifying, in some sense, how strongly the norm is spread over the coordinates.
Probability (ii) indicates how likely it is that the first $p$ components are all large, while this is not true for any of the other components, given that the norm of the vector is large. Probability (iii) specifies how likely it is that the first component is extreme, given that any component is extreme. In a financial context, such probabilities are used to quantify how strongly a specific market participant is exposed to a failure of any market participant. Finally, probability (iv) specifies the minimal contribution of any coordinate to the norm. Note that under Hypothesis 1 this probability equals 0.
The other true values are determined by Monte Carlo simulations with sample size of at least $10^7$, unless they can be easily calculated analytically; the approximation error is always smaller than $ 10^{-3}$. Throughout, we assume $\alpha$ to be known since we are interested in the effect of the PCA procedure on the estimator of the spectral measure, which should not be compounded with the estimation error of the tail index.

We consider different models of $d$-dimensional regularly varying vectors for which the spectral measure is (approximately) concentrated on a $p$-dimensional subspace. Since PCA is equivariant under rotations, w.l.o.g.\ we may assume that this subspace is spanned by the first $p$ unit vectors.

Two different models for the extreme value dependence structure between the first $p$ coordinates of the vector are investigated.  First, we consider the so-called Dirichlet model; see, for instance, \cite{segers2012}, Ex.\ 3.6, where also a simple algorithm is given to simulate vectors with such an extremal dependence structure. Second, we simulate random vectors with a Gumbel copula $C_\vartheta(x)=\exp\big(-\big(\sum_{i=1}^p(-\log x_i)^\vartheta\big)^{1/\vartheta}\big)$, using the transformation method proposed by \cite{stephenson2003}.   The marginal distributions are chosen as a Fr\'{e}chet distribution with cdf $\exp(-x^{-\alpha})$, $\alpha\in\{1,2\}$.

In addition, we have simulated observations from a Dirichlet model which are then rotated in the plane spanned by two randomly chosen coordinates, one of them among the first $p$ coordinates, the other among the last $d-p$. The rotation angle is uniformly distributed on the interval $[-\pi/10,\pi/10]$. Note that, unlike in the first two models, Hypothesis 1 is not fulfilled here which allows to evaluate how sensitive PCA is to moderate deviations from this ideal situation.

In all cases, we add the modulus of a $d$-dimensional multivariate normal vector with suitable variances  and constant correlations $0.2$. This way, it is ensured that the support of the exceedances  over high thresholds is not fully concentrated on the $p$-dimensional subspace. The variances are chosen equal to $10^5/d$ for $\alpha=1$ (i.e., if we start with unit Fr\'{e}chet margins) and equal to $10/d$ for $\alpha=2$, so that the sparsity assumption becomes apparent for the most extreme observations, whereas large yet less extreme data points are more spread out.

In all settings, we simulate samples of size $n=1000$ and examine the performance of the PCA procedure based on $\ang=X/\|X\|$ for the $k$ vectors with largest norms for $k\in\{5,10,15,\ldots,200\}$. The results reported here are based on 1000 simulations in each setting.

\begin{figure}
\vspace*{-7.5cm}
\hspace*{0.5cm}
\includegraphics[width=\textwidth]{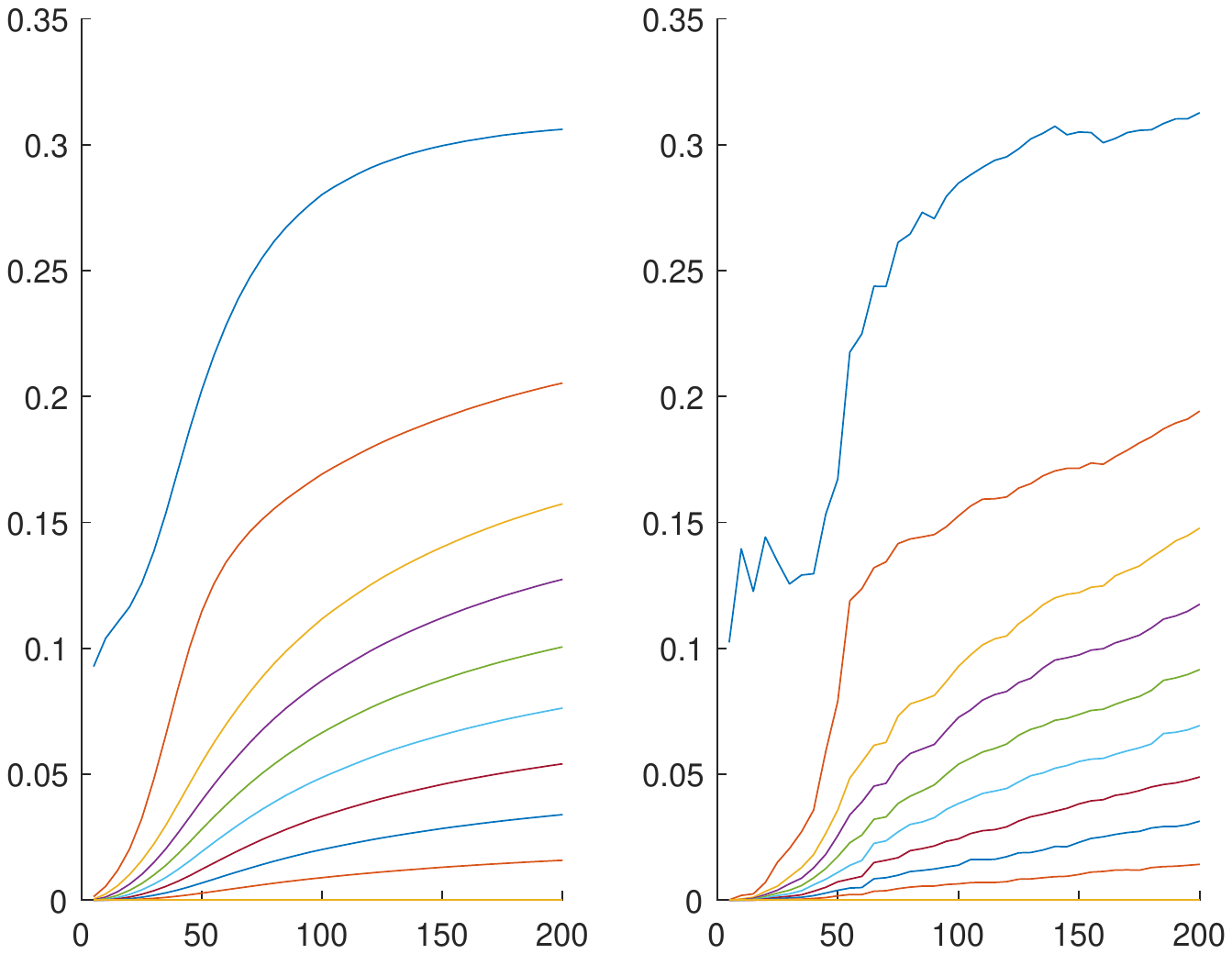} \vspace*{-6.2cm}
\caption{Mean empirical risk (left) and empirical risk for one sample (right) versus $k$ for PCA projecting onto a subspace of dimension $1\le \tilde p\le 10$ in the Dirichlet model with parameter $3$, $p=2$ and $d=10$}
 \label{fig:empriskdiri3_10_2}
\end{figure}

We first discuss the simulation results for the Dirichlet model with all Dirichlet parameters $\alpha_i$, $1\le i\le p$, equal to 3 and unit Fr\'{e}chet margins. Figure \ref{fig:empriskdiri3_10_2} shows the mean empirical risk in the left plot  as a function of $k$ for the PCA which projects onto a $\tilde p$-dimensional subspace with $1\le \tilde p\le 10$; here the true $p$ equals 2 and the vectors have dimension $d=10$. Since the mean empirical risk cannot be observed if one analyzes a given data set, the right plot shows the corresponding empirical risk for a single data set. The structure of both plots is very similar: essentially, the mean empirical risk curves are just a bit smoother. For this reason,  in the remaining settings, we will only report the mean empirical risk.

It is obvious from the risk plot that $\tilde p=2$ is a good choice, since there is a big gap to the empirical risk for $\tilde p=1$, whereas the empirical risk almost vanishes for small $k$ and $\tilde p=2$, and the risk decreases more regularly for values $\tilde p>2$, with no obvious structural breaks. The growing influence of the multivariate normal component as $k$ increases is manifest in these plots, since the empirical risk quickly increases with $k$ for all choices of $\tilde p$. This suggests to choose $k$ rather small to detect the sparsity in the model, a finding which will be corroborated in the analysis of the estimator of the spectral measure below.

In Figure \ref{fig:errdiri3_10_2}, the mean operator norm of the difference between the projection onto the true support of the limit measure $\mu$ and the projection onto the subspace of dimension 2 chosen by PCA is plotted versus $k$. Again it becomes obvious that for less extreme observations the approximation by a lower-dimensional vector is rather poor, which leads to a larger error for the projection matrix estimated from these data. For $k=80$, the norm has almost reached its maximal value. However, one should keep in mind that the operator norm measures the maximal distance between the projection of some vector $y\in\sphere$ onto the estimated respectively the true subspace. If the underlying distribution of $X/\|X\|$ puts little mass on vectors $y$ for which the distance is large, the true risk corresponding to the estimated subspace may still be small. 

\begin{figure}
\vspace*{-6.5cm}
\hspace*{1.5cm}
\includegraphics[width=0.8\textwidth]{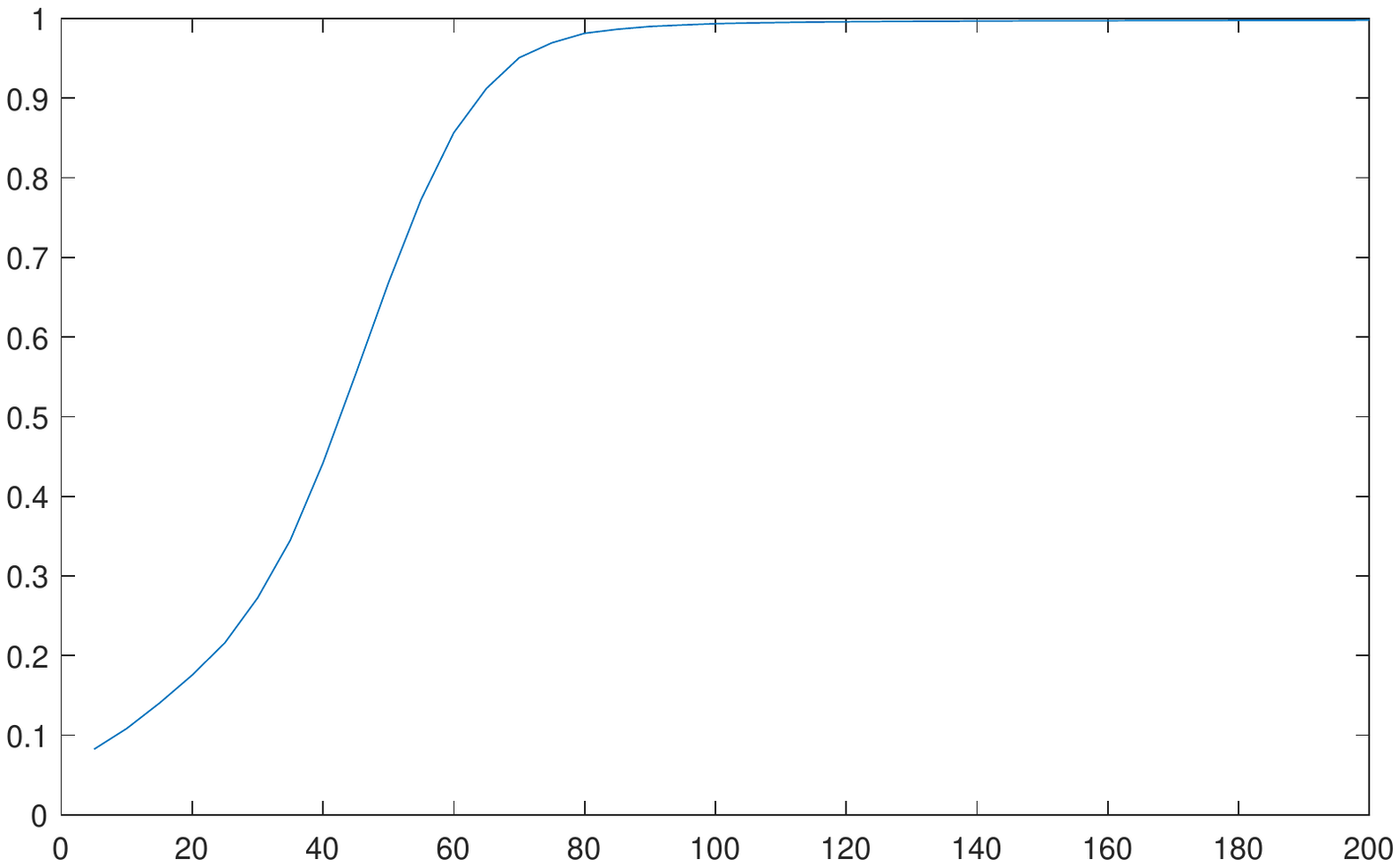} \vspace*{-4.8cm}
\caption{Mean operator norm of the difference between the projection onto the true subspace and the projection onto the two-dimensional subspace picked by PCA  as a function of $k$ in the Dirichlet model with parameter $3$, $p=2$ and $d=10$}
 \label{fig:errdiri3_10_2}
\end{figure}

Next we consider the estimators of the probabilities (i)--(iv), obtained by replacing the spectral measure $H$ either with $\hat H_{n,k}$ or $\hat H_{n,k}^{PCA}$. Since the PCA estimator of the subspace supporting $\mu$ quickly deteriorates as $k$ increases, in addition we consider the estimators resulting from $\hat H_{n,k,10}^{PCA}$, that uses just the largest 10 observations to estimate the supporting subspace.

\begin{figure}
\vspace*{-7.5cm}
\includegraphics[width=\textwidth]{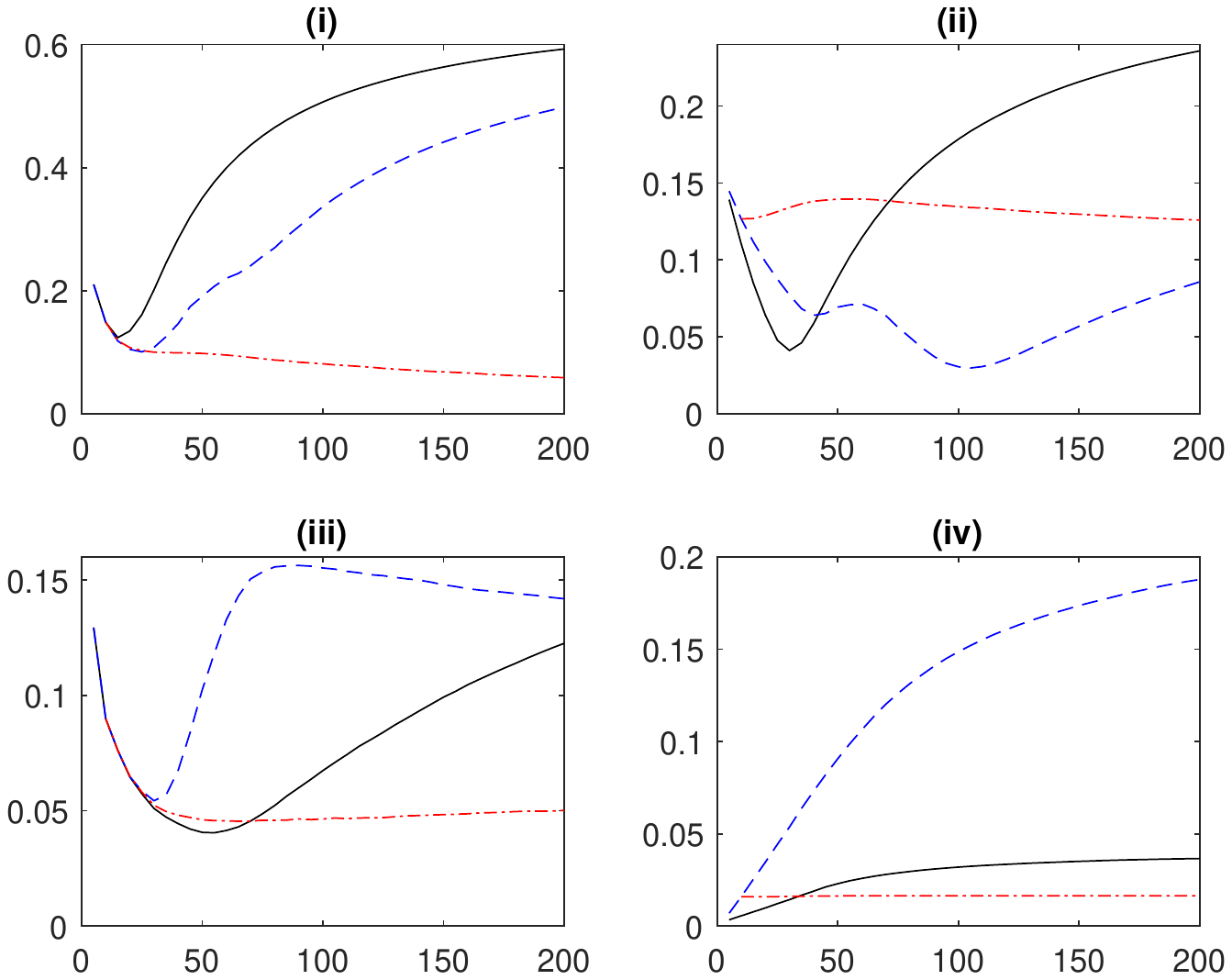} \vspace*{-6.2cm}
\caption{RMSE of the estimators of the probabilities (i)--(iv) based on $\hat H_{n,k}$ (black, solid), $\hat H_{n,k}^{PCA}$ (blue, dashed) and $\hat H_{n,k,10}^{PCA}$ (red, dash-dotted) versus $k$ in the Dirichlet model with parameter $3$, $p=2$ and $d=10$}
 \label{fig:statsdiri3_10_2}
\end{figure}

Figure \ref{fig:statsdiri3_10_2} displays the root mean squared errors (RMSE) of the resulting estimators as a function of $k$. For very small values of $k$, all estimators perform similarly. For probability (i)  with $t_{(i)}=0.65$ (leading to a true value of about 0.684), both PCA based estimators have a considerably smaller RMSE than the standard estimator for most $k$. In particular, the PCA based method using just 10 largest observations to estimate the support of the spectral measure clearly outperforms both other estimators (almost) irrespective of the number of observations used for estimation of the spectral measure.


For the  estimation of  probability (ii) ($\approx 0.309$), the standard non-parametric estimator performs best for $k\le 40$. The classical PCA using the same number of order statistics in both steps performs better for larger values of $k$ and its minimum RMSE is a bit lower than that of the standard estimator. The PCA based estimator which determines the support of $\mu$ from the largest 10 observations has a very stable RMSE, but its minimum is much larger than that both of the other estimators. 

In case (iii) (with true value of about $0.770$), the RMSE of the standard estimator and the estimator based on $\hat H_{n,k,10}^{PCA}$ are very similar for $k$ up to about 80, but the latter is remarkably insensitive to the choice of $k$ up to 200. This feature might be useful in practical applications where the selection of $k$ is often tricky. In contrast,
 the PCA based procedure which uses the same number of largest observations in both steps is even more sensitive to this choice than the standard estimator.

 Similarly, the classical PCA estimator of probability (iv) strongly depends on the choice of $k$ while both other estimators  stably have a very low error.

\begin{figure}
\vspace*{-7.5cm}
\includegraphics[width=\textwidth]{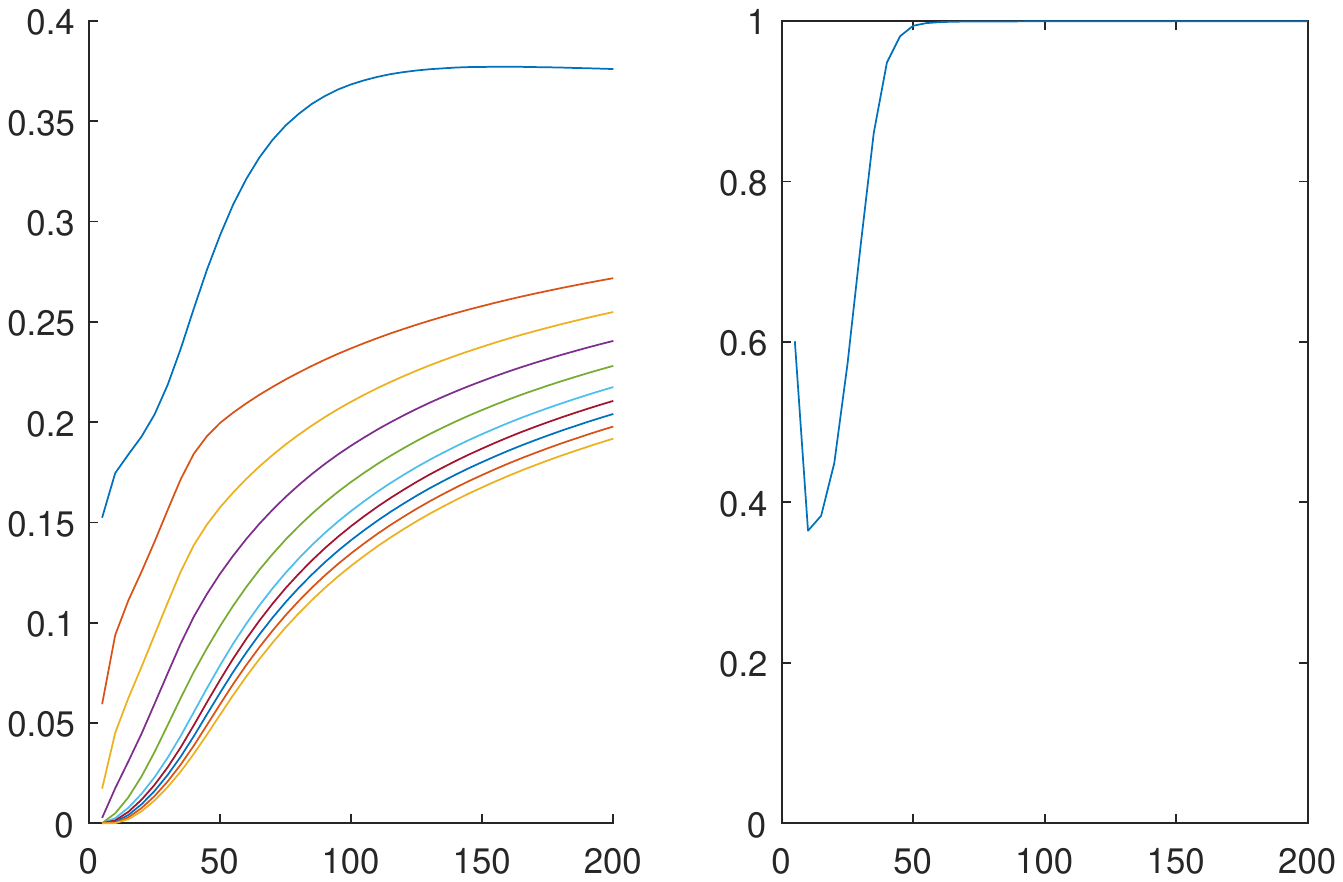} \vspace*{-6.8cm}
\caption{Mean empirical risk for PCA projecting onto a subspace of dimension $1\le \tilde p\le 10$  (left) and mean operator norm of the difference between the projection onto the true subspace and the projection onto the subspace picked by PCA with $\tilde p=p$ (right) versus $k$ in the Dirichlet model with parameter $3$, $p=5$ and $d=100$}
 \label{fig:empriskdiri3_100_5}
\end{figure}

Next, we consider the Dirichlet model with total dimension $d=100$ when the limit measure is concentrated on a $p=5$ dimensional subspace. Figure \ref{fig:empriskdiri3_100_5} shows the mean empirical risk for PCA projecting on a $\tilde p\in\{1,\ldots,10\}$ dimensional subspace in the left plot and the mean operator norm of the difference between the estimated and the true projection matrix in the right plot. The empirical risk suggests to choose $\tilde p$ between 4 and 6 and $k$ not much larger than 50 for estimating the support of the limit measure.

\begin{figure}
\vspace*{-6.8cm}
\hspace*{-0.4cm}
\includegraphics[width=1.1\textwidth]{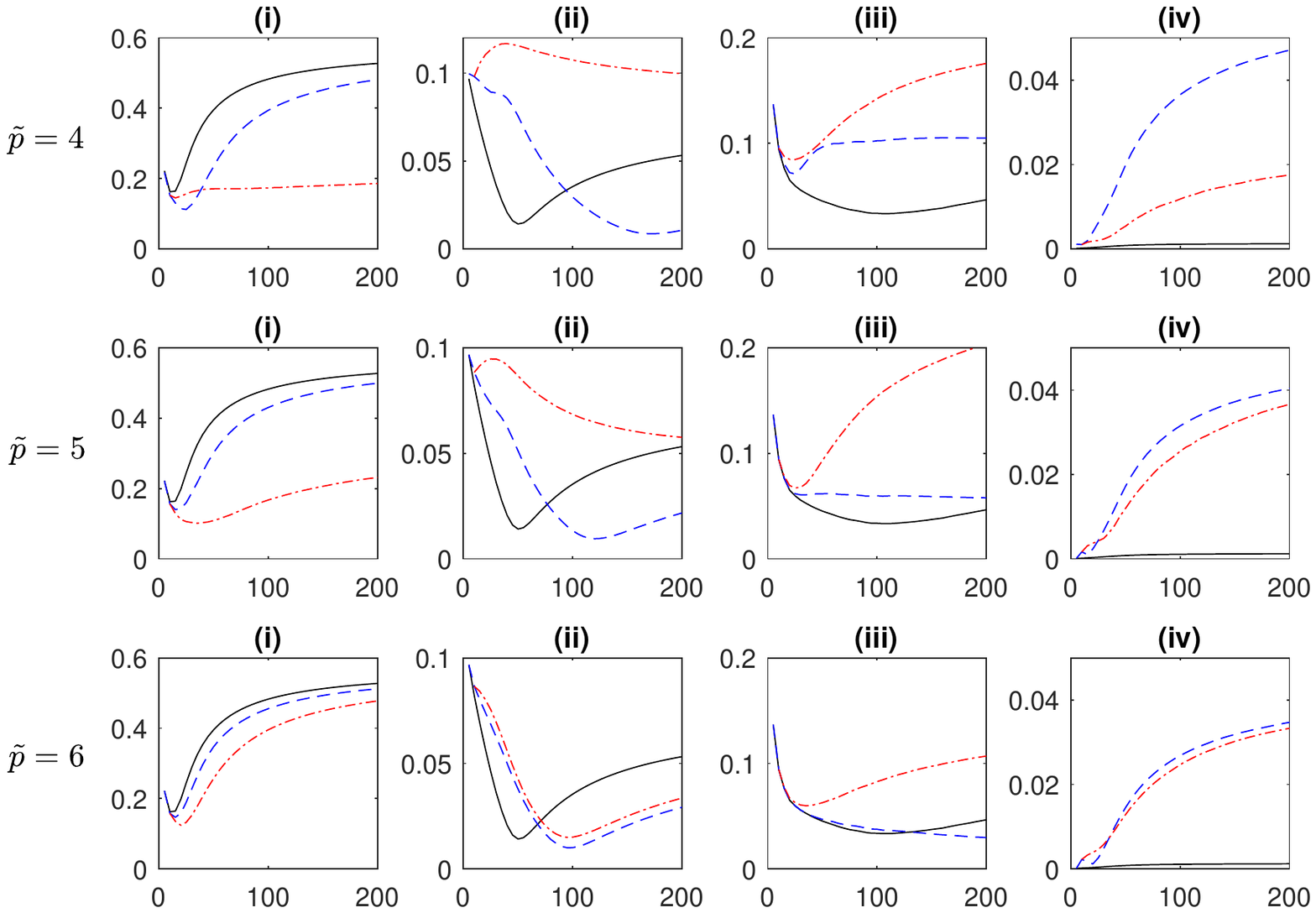} \vspace*{-6.5cm}
\caption{RMSE of the estimators of the probabilities (i)--(iv) based on $\hat H_{n,k}$ (black, solid), $\hat H_{n,k}^{PCA}$ (blue, dashed) and $\hat H_{n,k,10}^{PCA}$ (red, dash-dotted) vs.\ $k$ in the Dirichlet model with parameter $3$, $p=5$ and $d=100$; the upper plots correspond to PCA projections on subspaces of dimension $\tilde p=4$, the middle to $\tilde p=5$, and the lower to $\tilde p=6$}
 \label{fig:statsdiri3_100_5}
\end{figure}

Figure \ref{fig:statsdiri3_100_5} shows the RMSE of the different estimators of the probabilities (i)--(iv) with $t_{(i)}=0.4$ and true values $0.573, 0.072, 0.584$ and 0, respectively. Here, we have used PCA with $\tilde p=4$ in the upper row, $\tilde p=5$ in the mid row and $\tilde p=6$ in the lower row. As expected, in most cases the PCA procedures perform worse when they project on too low dimensional subspaces, yet in the cases (i) and (iv) the differences are moderate. At first glance somewhat surprisingly, overall the PCA procedures exhibit a better behavior for $\tilde p=6$ than for the ``correct'' value $\tilde p=5$. This may be explained by the fact that the extra dimension offers the opportunity to compensate for the difference between the subspaces minimizing the true resp.\ the empirical risk. This difference is expected to be larger if the dimension of the observed vectors is large, as can also be seen from the right plot in Figure \ref{fig:empriskdiri3_100_5}.

Again, the PCA based estimators for probability (i) outperform the standard procedure, but the other probabilities are more accurately estimated by the standard procedures if $\tilde p\le 5$ (though all estimators of (iv) perform reasonably well). For $\tilde p=6$, the RMSE of both variants of PCA based estimators of (ii) are very similar with a minimum value that is somewhat smaller than the minimum RMSE of the standard estimator. The performance of the standard estimator and the one based on classical PCA are almost identical for the probability (iii), while the estimator with PCA based on just $k=10$ largest observations is less accurate, probably because it is difficult to estimate a subspace of dimension 6 based on just 10 observations. It might help to increase the number of largest observations used to estimate the supporting subspace with the dimension $d$, but we do not explore this idea here in order not to overload the presentation.

\begin{figure}
\vspace*{-7.9cm}
\hspace*{0.5cm}
\includegraphics[width=0.9\textwidth]{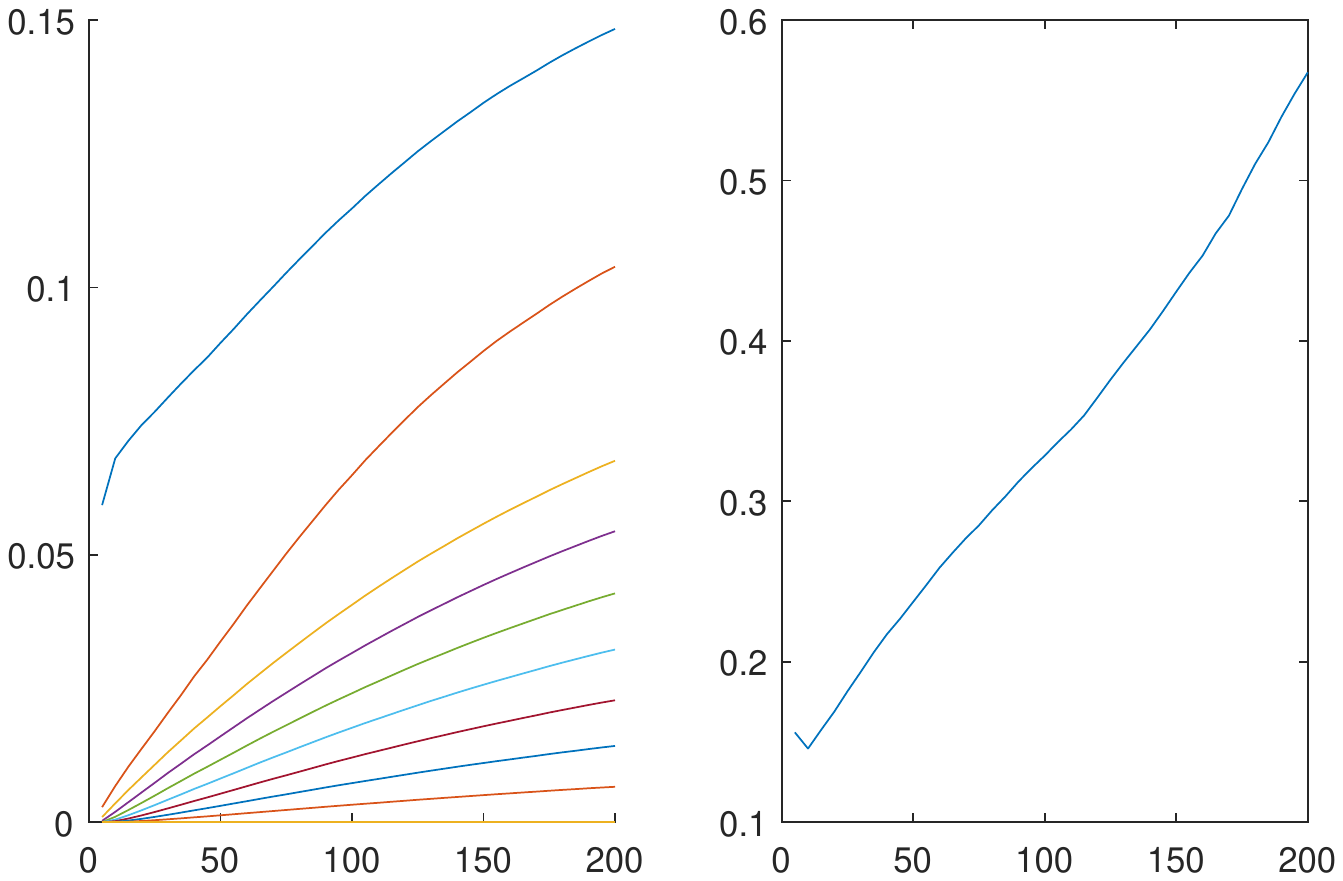} \vspace*{-5.5cm}
\caption{Mean empirical risk for PCA projecting onto a subspace of dimension $1\le \tilde p\le 10$  (left) and mean operator norm of the difference between the projection onto the true subspace and the projection onto the subspace picked by PCA with $\tilde p=p$ (right) versus $k$ in the Gumbel model with parameter $\vartheta=2$, $p=2$ and $d=10$}
 \label{fig:empriskGum05_10_2}
\end{figure}

The mean empirical risk and the mean operator norm of the difference matrix are shown for the Gumbel copula with $\vartheta=2$, $d=10$ and $p=2$ in Figure \ref{fig:empriskGum05_10_2}. Here, we have chosen Fr\'{e}chet marginal distributions with cdf $F(x)=\exp(-x^{-2})$, $x>0$. Based on the left plot, one may choose $\tilde p=2$, or perhaps $\tilde p=3$.

\begin{figure}
\vspace*{-7.5cm}
\includegraphics[width=\textwidth]{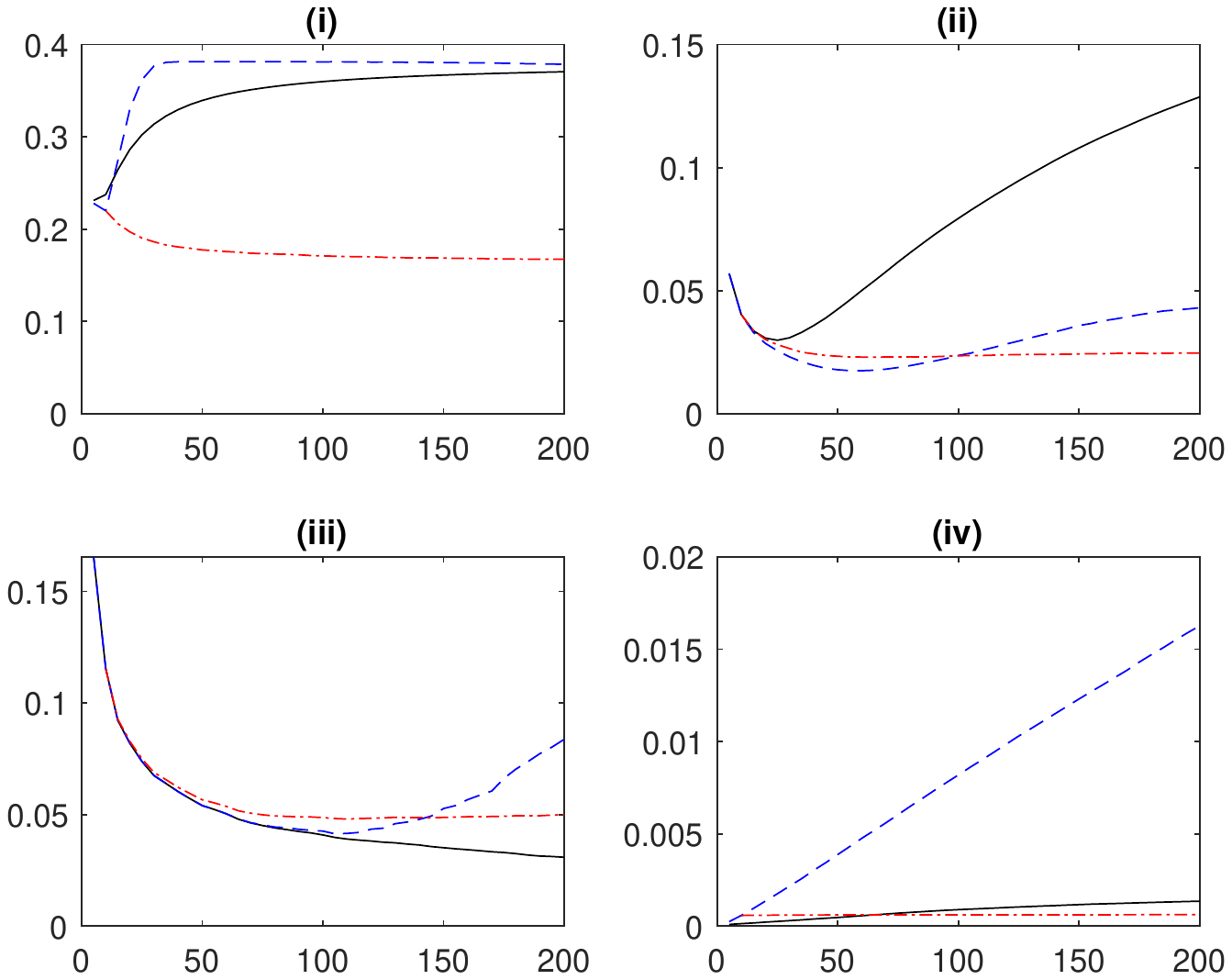} \vspace*{-6.3cm}
\caption{RMSE of the estimators of the probabilities (i)--(iv) based on $\hat H_{n,k}$ (black, solid), $\hat H_{n,k}^{PCA}$ (blue, dashed) and $\hat H_{n,k,10}^{PCA}$ (red, dash-dotted) vs.\ $k$ in the Gumbel model with parameter $\vartheta=2$, $p=2$ and $d=10$}
 \label{fig:statsGum05_10_2}
\end{figure}

Figure \ref{fig:statsGum05_10_2} displays the RMSE of the estimators of (i)--(iv) with PCA projecting on two-dimensional subspaces. Here $t_{(i)}=0.7$ and the true values for (i)--(iv) are 0.3813, $0.083$, $1/\sqrt{2}$ and 0. Now the PCA which uses the same number in both steps performs worse than the standard estimator for probability (i), better for (ii)  and very similar to the standard procedure for  (iii) and $k\le 100$. The PCA estimator which uses just 10 largest observations for estimating the support again outperforms the standard procedure for probability (i) and (ii), whereas it is has a slightly larger RMSE for the probability (iii). In any case, as in the Dirichlet model, its RMSE is rather insensitive to the choice of $k$. If one chooses $\tilde p=3$ (plots not shown here), then the classical PCA has almost the same RMSE as the standard procedure for the probabilities (i) and (iii). The same holds true for the estimator based on $\hat H_{n,k,10}^{PCA}$ for (iii) and (iv), while for (i) this estimator is still considerably more accurate than the standard procedure (though less so than for $\tilde p=2$) and both PCA procedures still outperform the standard estimator for probability (ii).

For the high-dimensional Gumbel model with $d=100$ and $p=5$, by and large the findings are similar to the ones observed for the Dirichlet model so that we do not show the corresponding plots. However, in this model $\tilde p=4$ can be ruled out by the empirical risk plot and both PCA based estimators outperform the standard estimator of (ii).

\begin{figure}
\vspace*{-6.5cm}
\includegraphics[width=\textwidth]{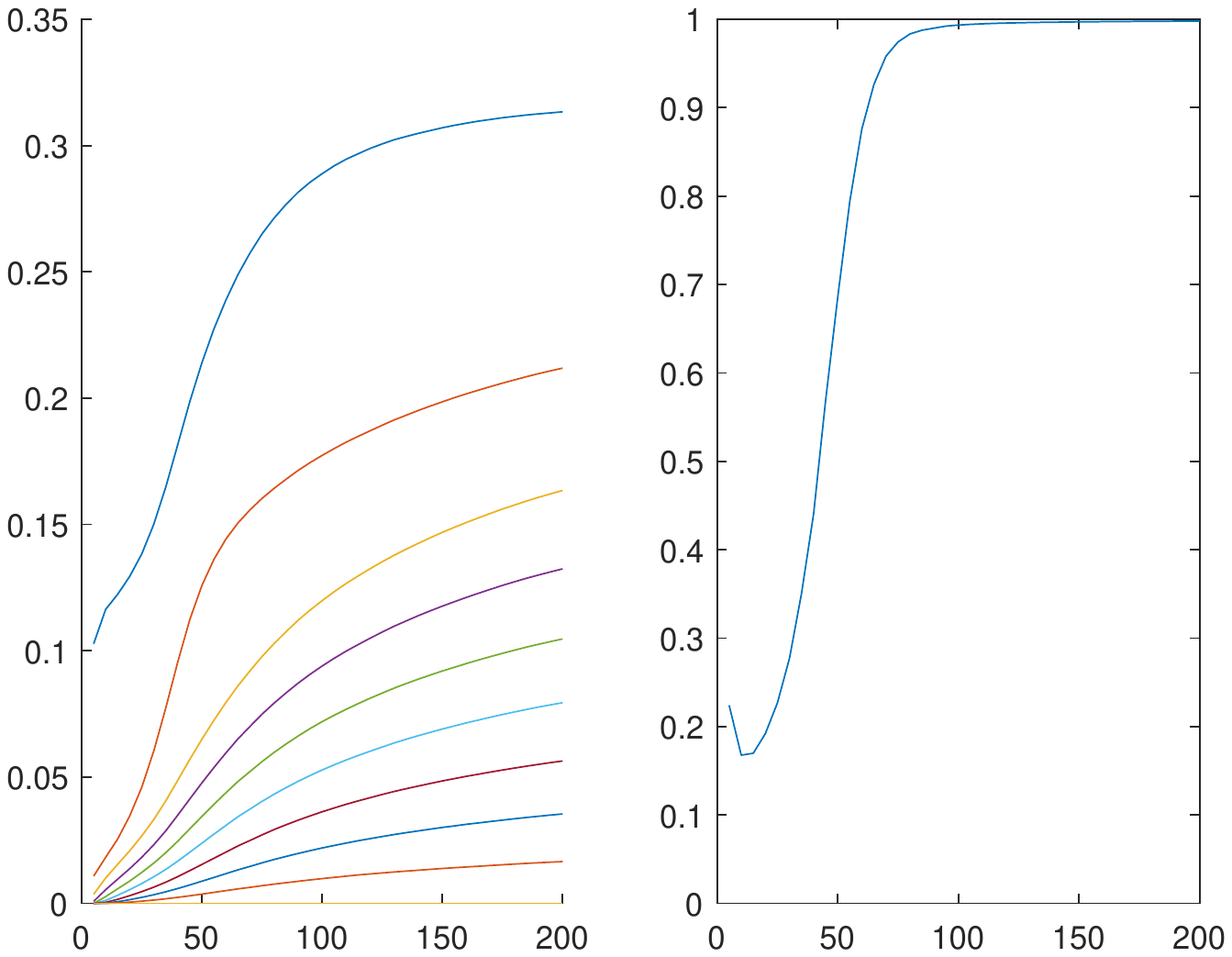} \vspace*{-6.0cm}
\caption{Mean empirical risk for PCA projecting onto a subspace of dimension $1\le \tilde p\le 10$  (left) and mean operator norm of the difference between the projection onto the true subspace and the projection onto the subspace picked by PCA with $\tilde p=p$ (right) versus $k$ for randomly rotated Dirichlet observations with parameter $3$, $p=2$ and $d=10$ }
 \label{fig:empriskdiri_3_2_10rot}
\end{figure}

Finally, we turn to the disturbed Dirichlet model with $d=10$ and $p=2$ where the observations are randomly rotated by an angular up to $\pi/10$, leading to true values for (i)--(iv) of 0.653 (with $t_{(i)}=0.65$), 0.185, 0.770 and 0. The corresponding plots are shown in the Figures \ref{fig:empriskdiri_3_2_10rot} and \ref{fig:statsdiri_3_2_10rot}. In view of the empirical risk, the choices $\tilde p\in\{2,3\}$ seem reasonable.

Again, the PCA procedure which uses the same largest observations in both steps performs better for the larger choice of $\tilde p$, whereas the performance of the other PCA procedure improves only for (ii), while it does not change much for (iii) and it deteriorates a bit for (i) and (iv). The PCA estimators perform better than the standard procedure for probability (i) and for (iii) if $k$ is large (for classical PCA only if $\tilde p=3$), whereas for (ii), roughly speaking,  overall the estimators perform similarly well with the standard procedure performing better for small values of $k$ and the PCA estimators for larger values.

\begin{figure}
\vspace*{-7.5cm}
\includegraphics[width=\textwidth]{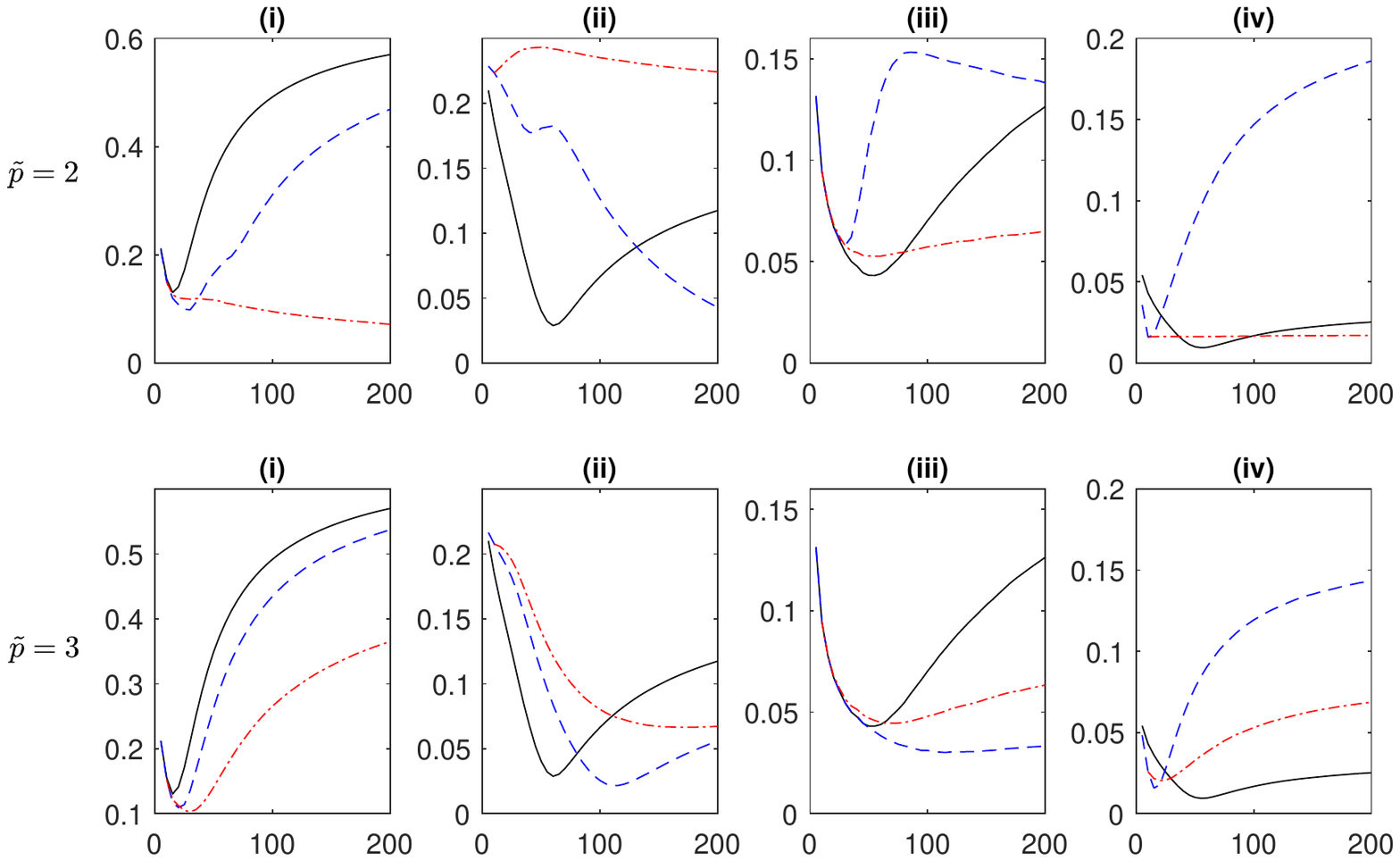} \vspace*{-6.5cm}
\caption{RMSE of the estimators of the probabilities (i)--(iv) based on $\hat H_{n,k}$ (black, solid), $\hat H_{n,k}^{PCA}$ (blue, dashed) and $\hat H_{n,k,10}^{PCA}$ (red, dash-dotted) vs.\ $k$ for randomly rotated Dirichlet observations with parameter $3$, $p=2$ and $d=10$; the upper plots correspond to PCA projections on subspaces of dimension $\tilde p=2$, the lower to $\tilde p=3$}
 \label{fig:statsdiri_3_2_10rot}
\end{figure}

To sum up, the plot of the empirical risk seems a useful tool to choose the dimension of the subspace onto which the PCA procedure projects. In particular, for the PCA method which uses the same number of largest observations to estimate the support and to calculate the estimator of the spectral measure, in case of doubt it seems advisable to project onto a subspace of higher dimension. While the PCA step does not always improve the estimator of the spectral measure, in most cases the resulting estimators seem competitive with the standard ones if $\tilde p$ is chosen appropriately, and for probability (i) they are superior. In particular the PCA estimators which determine the support based only on the largest 10 observations often exhibit a desirable insensitivity to the choice of largest observations used to estimate the spectral measure.

\section{Appendix: Details of the proof of \eqref{eq:Blanchardbound}}
\label{app:proof}

Recall that $Z_i$ are iid random variables  whose distribution equals the conditional distribution of $\ang$ given $\|X\|>t$.
Let
$$ \phi^\pm(z_1,\ldots,z_\ell) := \sup_{V\in\VV_p} \pm \Big( \frac 1\ell \sum_{i=1}^\ell \|\proj_V^\perp z_i\|^2-P\|\proj_V^\perp Z_1\|^2\Big).
$$
First note that
$$ \big|\phi^\pm(z_{1:\ell})-\phi^\pm(z_{1:i-1},\tilde z_i,z_{i+1:\ell})\big| \le \sup_{V\in\VV_p}\frac 1\ell \big|\|\proj_V^\perp z_i\|^2-\|\proj_V^\perp \tilde z_i\|^2\big| \le \frac 1\ell
$$
for all $z,\tilde z\in B_1(0)$. Thus a simple version of the bounded difference inequality (see, e.g., Theorem 3.1 of \cite{mcdiarmid1998concentration}) gives
$$ \PP\big\{ \phi^\pm(Z_{1:\ell})- \esp \phi^\pm(Z_{1:\ell})\ge u\big\} \le\exp(-2\ell u^2), \quad \forall u>0.
$$
Exactly in the same way as in the proof of Lemma \ref{lem:expectationmaxdiff}, one obtains
\begin{align*}
 \esp \phi^\pm(Z_{1:\ell}) & \le \Big[ \frac{p\wedge (d-p)}\ell \esp\|ZZ^\top-\esp ZZ^\top\|_{HS}^2\Big]^{1/2}\\
  & =  \Big[ \frac{p\wedge (d-p)}\ell \big(\esp\|ZZ^\top\|_{HS}^2- \|\esp ZZ^\top\|_{HS}^2\big)\Big]^{1/2}.
\end{align*}
Let $\tilde Z$ be an independent copy of $Z$. Then
$$ \esp\|ZZ^\top-\tilde Z \tilde Z^\top\|_{HS}^2 = 2\esp\|Z Z^\top\|_{HS}^2 - 2\esp \langle ZZ^\top,\tilde Z\tilde Z^\top\rangle_{HS}
$$
with
$$ \esp \langle ZZ^\top,\tilde Z\tilde Z^\top\rangle_{HS} = \esp\big( \esp(\langle ZZ^\top,\tilde Z\tilde Z^\top\rangle_{HS}\mid Z)\big) = \esp\langle ZZ^\top,\esp \tilde Z\tilde Z^\top\rangle_{HS}
= \|\esp ZZ^\top\|_{HS}^2.
$$
To sum up, so far we have shown that
$$ \PP\Big\{ \phi^\pm(Z_{1:\ell})\ge \Big[ \frac{p\wedge (d-p)}{2\ell}  \esp\|ZZ^\top-\tilde Z \tilde Z^\top\|_{HS}^2\Big]^{1/2}+u\Big\} \le \exp(-2\ell u^2), \quad \forall u\ge 0.
$$

Next consider the U-statistic $U:=(\ell(\ell-1))^{-1} \sum_{i,j=1}^\ell g(Z_i,Z_j)$ with
$$ g(z,\tilde z) := \|zz^\top-\tilde z\tilde z^\top\|_{HS}^2 \le (\|zz^\top\|_{HS}+\|\tilde z\tilde z^\top\|_{HS})^2\le 4.
$$
By equation (5.7) of \cite{hoeffding1963}, one has
$$ \PP\{U-\esp U\ge 2v\} \le \exp\big(-2\floor{\ell/2} (2v)^2/16\big)= \exp\big(-\floor{\ell/2}v^2/2\big),\quad \forall v\ge 0,
$$
with $\esp U=\esp\|ZZ^\top-\tilde Z\tilde Z^\top\|_{HS}^2$. Hence,
\begin{align*}
  \PP \Big\{ & \max\big(\phi^+(Z_{1:\ell}),\phi^-(Z_{1:\ell})\big) \ge \Big[ \frac{p\wedge (d-p)}{2\ell} \Big( \frac 1{\ell(\ell-1)} \sum_{i,j=1}^\ell \|Z_iZ_i^\top-Z_jZ_j^\top\|_{HS}^2 +2v\Big)\Big]^{1/2}\\
   & { } +u\Big\}
   \le 2\exp(-2\ell u^2)+\exp\big(-\floor{\ell/2}v^2/2\big), \quad \forall u,v\ge 0.
\end{align*}
This, however, is equivalent to \eqref{eq:Blanchardbound}, because the joint distribution of
$\max\big(\phi^+(Z_{1:\ell}),$ $\phi^-(Z_{1:\ell})\big)$ and $\sum_{i,j=1}^\ell \|Z_iZ_i^\top-Z_jZ_j^\top\|_{HS}^2$ is the same as the joint conditional distribution of $\sup_{V\in\VV_p}|\hat R_t(V)-R_t(V)|$ and $\sum_{i,j=1}^\ell  \|\ang_{(i)} \ang_{(i)}^\top-\ang_{(j)}\ang_{(j)}^\top\|_{HS}^2$, given $N_t=\ell$.

\bigskip

{\bf Acknowledgements:} Holger Drees was partly supported by DFG grant DR271/6-2. Anne Sabourin was partly supported by the Chaire `Stress testing' from \'{E}cole Polytechnique and BNP Paribas.

\bibliographystyle{apalike}
\bibliography{PCA-extremes}

\end{document}